\documentclass[times]{article}
\usepackage{mystyle}

\usepackage{moreverb}
\usepackage{amsmath,amscd,bm,subcaption,xfrac}

\usepackage[utf8]{inputenc}
\usepackage[english]{babel}

\newtheorem{lemma}{Lemma}
\usepackage{algorithm}
\usepackage{algpseudocode}
\usepackage{booktabs}
\newcommand{\tabsize}{\fontsize{9}{9.5pt}\selectfont}

\begin{document}

\title{A geometric multigrid method for isogeometric compatible discretizations of the generalized Stokes and Oseen problems}

\author{Christopher Coley,\footnote{Ann and H.J. Smead Aerospace Engineering Sciences, 426 UCB, University of Colorado Boulder 80309, USA} \footnote{ E-mail: christopher.coley@colorado.edu} \ Joseph Benzaken\footnote{Department of Applied Mathematics, 526 UCB, University of Colorado, Boulder, CO 80309, USA}, \ John A. Evans$^*$}

\maketitle

\begin{abstract}
In this paper, we present a geometric multigrid methodology for the solution of matrix systems associated with isogeometric compatible discretizations of the generalized Stokes and Oseen problems.  The methodology provably yields a pointwise divergence-free velocity field independent of the number of pre-smoothing steps, post-smoothing steps, grid levels, or cycles in a V-cycle implementation.  The methodology relies upon Scwharz-style smoothers in conjunction with specially defined overlapping subdomains that respect the underlying topological structure of the generalized Stokes and Oseen problems.  Numerical results in both two- and three-dimensions demonstrate the robustness of the methodology through the invariance of convergence rates with respect to grid resolution and flow parameters for the generalized Stokes problem as well as the generalized Oseen problem provided it is not advection-dominated.
\\\\
\noindent Keywords: Geometric multigrid; Isogeometric compatible discretizations; Isogeometric divergence-conforming discretizations; Generalized Stokes flow; Generalized Oseen flow; Overlapping Schwarz smoothers

\end{abstract}

\section{Introduction}

Isogeometric compatible discretizations\footnote{Depending on context, isogeometric compatible discretizations may also be referred to as isogeometric discrete differential forms, structure-preserving discretizations, divergence-conforming discretizations, or curl-conforming discretizations.} have recently arisen as an attractive candidate for the spatial discretization of fluid flow problems \cite{Buffa11,Evans12,Evans13_1,Evans13_2,Evans13_3,Evans16,Sarmiento17}.  These discretizations comprise a discrete Stokes complex \cite{Evans11,Falk13} and may be interpreted as smooth generalizations of Raviart-Thomas-N\'{e}d\'{e}lec finite elements \cite{Raviart77,Nedelec80}.  When applied to incompressible flow problems, isogeometric compatible discretizations produce pointwise divergence-free velocity fields and hence exactly satisfy mass conservation.  As a result, they preserve the balance law structure of the incompressible Navier-Stokes equations, and in particular, they properly conserve mass, linear and angular momentum, energy, vorticity, enstrophy (in the two-dimensional setting), and helicity (in the three-dimensional setting) in the inviscid limit \cite{Evans13_3}.  Isogeometric compatible discretizations have recently been applied to Cahn-Hilliard flow \cite{Vignal15}, turbulent flow \cite{VanOpstal17}, and fluid-structure interaction \cite{Kamensky17} where improved results were attained in comparison with state-of-the-art discretization procedures.

Despite the promise of isogeometric compatible discretizations, very little research has been conducted so far in the area of efficient linear solvers.  In fact, only the performance of Krylov subspace methods in conjunction with block preconditioners has been investigated in prior work \cite{Cortes15,Cortes17}.  The objective of the current work is to introduce an optimally efficient linear solution procedure for isogeometric compatible discretizations of the generalized Stokes and Oseen problems.  It should be noted that there are many different candidates in this regard.  For instance, there exist efficient physics-based splitting methods such as the inexact Uzawa algorithm \cite{Elman94}.  However, these techniques rely on suitable Schur complement approximations which can be difficult to design in the context of generalized Oseen flow.  Alternatively, one can employ a multigrid method in conjunction with a Vanka smoother \cite{Vanka86}, a Uzawa smoother \cite{Gaspar14}, or a Braess-Sarazin smoother \cite{Braess97}.  While these techniques generally do not require accurate Schur complement approximations, they typically involve specially tuned relaxation parameters.  Perhaps more concerning is the fact that all of the aforementioned procedures do not return a pointwise-divergence free velocity field unless the linear solver is fully converged.

To overcome the issues associated with the aforementioned linear solution procedures, we present a geometric multigrid methodology which relies upon Schwarz-style smoothers \cite{Dolean15} in conjunction with specially defined overlapping subdomains that respect the underlying topological structure of the generalized Stokes and Oseen problems.  This methodology is inspired by multigrid and auxiliary space preconditioning methodologies for divergence-conforming discontinuous Galerkin formulations of Stokes flow \cite{AyusodeDios2014,Kanschat15,Kanschat16,Kanschat17} and multigrid methodologies for compatible finite element discretizations of Darcy and Maxwell problems \cite{Hiptmair98,Arnold00}.  We prove that our methodology yields a pointwise divergence-free velocity field independent of the number of pre-smoothing steps, post-smoothing steps, grid levels, or cycles in a V-cycle implementation.  We also demonstrate by numerical example that our methodology is optimally efficient and robust in that it exhibits convergence rates independent of the grid resolution and flow parameters for the generalized Stokes problem as well as the generalized Oseen problem provided it is not advection-dominated.  It should be mentioned that the only user-defined constants in our methodology are the number of pre-smoothing steps and post-smoothing steps as well as the scaling factor if one elects to use an additive Schwarz smoother rather than a multiplicative Schwarz smoother.  However, we have found that our method is optimally efficient irregardless of the number of pre- and post-smoothing steps selected.

An outline of the remainder of the paper is as follows.  In Section \ref{sec:NS},
we inspire the need for efficient linear solvers for the generalized Stokes and Oseen problems through a discussion of temporal discretization of the Navier-Stokes equations.  In Section \ref{sec:spatial}, we discuss spatial discretization of the generalized Stokes and Oseen problems.  In Section \ref{sec:DivConf}, we introduce the Stokes complex and demonstrate how to construct isogeometric compatible discretizations which commute with this complex.  In Section \ref{sec:multigrid}, we present our structure-preserving geometric multigrid methodology, and we prove that this methodology indeed yields discrete velocity fields which are divergence-free.  In Section \ref{sec:Results}, we apply the proposed multigrid method to a selection of generalized Stokes and Oseen problems.  Finally, in Section \ref{sec:Conclusion}, we provide concluding remarks.

\section{Temporal Discretization of the Navier-Stokes Equations and the Generalized Stokes and Oseen Problems}
\label{sec:NS}

To motivate the need for efficient linear solvers for the generalized Stokes and Oseen problems, we first demonstrate how such problems arise through semi-implicit temporal discretization of the incompressible Navier-Stokes equations subject to homogeneous Dirichlet boundary conditions. For $d \in \mathbb{Z}_+$, let $\Omega \subset \mathbb{R}^d$ denote an open, Lipschitz bounded domain, let $\Gamma$ denote the boundary of $\Omega$, and let $T \in \mathbb{R}_+$. Given $\nu \in \mathbb{R}_+$, ${\bf f}: \Omega \times (0,T) \rightarrow \mathbb{R}^d$, and ${\bf u}_0: \Omega \rightarrow \mathbb{R}^d$, the strong form of the Navier-Stokes problem then reads as follows: Find ${\bf u}: \overline{\Omega} \times [0,T] \rightarrow \mathbb{R}^d$ and $p: \Omega \times (0,T) \rightarrow \mathbb{R}$ such that:
\begin{align}
\begin{array}{rll}
\frac{\partial {\bf u}}{\partial t} + {\bf u} \cdot \nabla {\bf u} - \nu \Delta{\bf u} + \nabla p &= {\bf f}& \textup{ for } ({\bf x},t) \in \Omega \times (0,T) \\
\nabla \cdot {\bf u} &= 0 & \textup{ for } ({\bf x},t) \in \Omega \times (0,T)\\
{\bf u} &= {\bf 0} & \textup{ for } ({\bf x},t) \in \Gamma \times (0,T)\\
\left. {\bf u} \right|_{t = 0} &= {\bf u}_0 & \textup{ for } {\bf x} \in \Omega
\end{array}
\label{eqn:NS}
\end{align}
Above, ${\bf u}$ denotes the velocity field, $p$ denotes the pressure field, $\nu$ denotes the kinematic viscosity, ${\bf f}$ denotes the force per unit mass, and ${\bf u}_0$ denotes the initial velocity field.  The velocity field is uniquely specified by the Navier-Stokes problem while the pressure field is unique up to a constant.  To discretize in time, we first define a sequence of time instances $t_0 < t_1 < t_2 < \ldots < t_N$ such that $t_0 = 0$ and $t_N = T$, and we denote the velocity and pressure solutions at the $n^{\textup{th}}$ time instance as ${\bf u}^{(n)}$ and $p^{(n)}$ respectively for $n = 0, \ldots, N$.  We further define $t_{n+1/2} = \frac{t_n + t_{n+1}}{2}$.  Without loss of generality, we assume that the time instances are equi-spaced, and we define the time step size to be $\Delta t = t_{n+1} - t_{n}$.  The velocity solution at $N = 0$ is given by ${\bf u}^{(0)} = {\bf u}_0$, while to find the velocity and pressure solutions at each subsequent time instance, we must discretize the Navier-Stokes problem in time.  We discuss two demonstrative semi-implicit temporal discretization schemes herein, though the proceeding discussion also applies to other semi-implicit temporal discretization schemes\footnote{With a fully implicit time discretization scheme, one must turn to a nonlinear solution procedure such as Newton's method.  However, with Newton's method, one solves a sequence of generalized Oseen problems.  Hence, the multigrid methodology discussed here can also be employed to solve these problems.}  

Let us first consider the standard Crank-Nicolson/Adams-Bashforth scheme \cite{He07}.  In this approach, a central difference approximation of the unsteady term and linear interpolation approximations of the diffusive and pressure force terms are employed:
\begin{align}
\frac{\partial {\bf u}}{\partial t}(t_{n+1/2}) &\approx \frac{{\bf u}^{(n+1)}-{\bf u}^{(n)}}{\Delta t} \nonumber \\
\Delta {\bf u}(t_{n+1/2}) &\approx \frac{\Delta {\bf u}^{(n+1)}+\Delta {\bf u}^{(n)}}{2} \nonumber \\
\nabla p(t_{n+1/2}) &\approx \frac{\nabla p^{(n+1)}+\nabla p^{(n)}}{2} \nonumber
\end{align}
The advection term at $t_{n+1/2}$ is alternatively approximated using Taylor-series expansions involving time instances $t_{n-1}$ and $t_{n}$, resulting in\footnote{This approximation is not properly defined for $n$ = 0.  Consequently, the approximation is replaced by ${\bf u}^{(n)} \cdot \nabla {\bf u}^{(n)}$ for $n$ = 0 in practice.}:
\begin{align}
\left({\bf u} \cdot \nabla {\bf u}\right)(t_{n+1/2}) &\approx \frac{3}{2} {\bf u}^{(n)} \cdot \nabla {\bf u}^{(n)} - \frac{1}{2} {\bf u}^{(n-1)} \cdot \nabla {\bf u}^{(n-1)} \nonumber
\end{align}
Collecting the above approximations, we find that the resulting generalized Stokes system holds for each $n = 0, \ldots, N-1$:
\begin{align}
\begin{array}{rll}
\sigma {\bf u}^{(n+1)} - \nu \Delta {\bf u}^{(n+1)} + \nabla p^{(n+1)} &= {\bf f}_{GS}^{(n+1)}& \textup{ for } {\bf x} \in \Omega \\
\nabla \cdot {\bf u}^{(n+1)} &= 0& \textup{ for } {\bf x} \in \Omega \\
{\bf u}^{(n+1)} &= {\bf 0}& \textup{ for } {\bf x} \in \Gamma
\end{array}
\label{eqn:GS}
\end{align}
where $\sigma = \sigma = \sfrac{2}{\Delta t}$ and:
\begin{align}
{\bf f}_{GS}^{(n+1)} = {\bf f}(t_{n+1/2}) + \sigma {\bf u}^{(n)} - 3 {\bf u}^{(n)} \cdot \nabla {\bf u}^{(n)} + {\bf u}^{(n-1)} \cdot \nabla {\bf u}^{(n-1)} + \nu \Delta {\bf u}^{(n)} - \nabla p^{(n)}
\end{align}
The above generalized Stokes system is reaction-dominated for small time step sizes and diffusion-dominated for large time step sizes.  We demonstrate later that our geometric multigrid methodology is robust for both of these regimes.

The advantage of the Crank-Nicolson/Adams-Bashforth scheme is that the advection term is handled in a purely explicit manner.  After spatial discretization, this leads to a symmetric matrix problem.  However, the disadvantage of the scheme is that it is stable only if the time-step is chosen sufficiently small as to satisfy a CFL condition.  With this in mind, we next consider an unconditionally stable semi-implicit scheme introduced by Guermond \cite{Guermond99}.  In this scheme, the unsteady, diffusive, and pressure force terms are approximated as before, but the advection term is approximated as follows\footnote{This approximation is also not properly defined for $n$ = 0.  Consequently, the approximation is replaced by ${\bf u}^{(n)} \cdot \nabla \left( \left({\bf u}^{(n)} + {\bf u}^{(n+1)}\right)/2 \right)$ for $n$ = 0 in practice.}:
\begin{align}
\left({\bf u} \cdot \nabla {\bf u}\right)(t_{n+1/2}) &\approx \left( \frac{3}{2} {\bf u}^{(n)} - \frac{1}{2} {\bf u}^{(n-1)} \right) \cdot \nabla \left( \frac{{\bf u}^{(n+1)} + {\bf u}^{(n)}}{2} \right) \nonumber
\end{align}
Note that the advection velocity is approximated in an explicit manner while the gradient is approximated in an implicit manner.  Collecting the above approximations, we find that the resulting generalized Oseen system holds for each $n = 0, \ldots, N-1$:
\begin{align}
\begin{array}{rll}
\sigma {\bf u}^{(n+1)} + {\bf a}^{(n+1)} \cdot \nabla {\bf u}^{(n+1)} - \nu \Delta {\bf u}^{(n+1)} + \nabla p^{(n+1)} &= {\bf f}_{GO}^{(n+1/2)}& \textup{ for } {\bf x} \in \Omega \\
\nabla \cdot {\bf u}^{(n+1)} &= 0& \textup{ for } {\bf x} \in \Omega \\
{\bf u}^{(n+1)} &= {\bf 0}& \textup{ for } {\bf x} \in \Gamma
\end{array}
\label{eqn:GO}
\end{align}
where $\sigma = \sfrac{2}{\Delta t}$, ${\bf a}^{(n+1)} = \frac{3}{2} {\bf u}^{(n)} - \frac{1}{2} {\bf u}^{(n-1)}$, and:
\begin{align}
{\bf f}_{GO}^{(n+1)} = {\bf f}(t_{n+1/2}) + \sigma {\bf u}^{(n)} - {\bf a}^{(n+1)} \cdot \nabla {\bf u}^{(n)} + \nu \Delta {\bf u}^{(n)} - \nabla p^{(n)}
\end{align}
In opposition with the generalized Stokes system obtained earlier, the above system admits different behavior based on not only the scalars $\sigma$ and $\nu$ but also the advection velocity ${\bf a}^{(n+1)}$.  We demonstrate later that our geometric multigrid methodology is robust for this system provided it is not advection-dominated.  This holds if a CFL-like condition is satisfied.

\section{Spatial Discretization of the Generalized Stokes and Oseen Problems}
\label{sec:spatial}

Now that we have motivated the need for efficient linear solvers for the generalized Stokes and Oseen problems, we turn to the question of spatial discretization.  In this section, we present the basic ingredients associated with a mixed Galerkin discretization.  Later, we will specialize to the setting of isogeometric compatible discretizations.

\subsection{Weak Formulation of the Generalized Stokes and Oseen Problems}

To begin, we must state a weak formulation for the generalized Stokes and Oseen Problems.  We strictly consider the case of homogeneous Dirichlet boundary conditions without loss of generality.  Before proceeding, we must first define suitable velocity and pressure test spaces:
\begin{align}
{\bf H}^1_0(\Omega) := \left\{ {\bf v} \in {\bf H}^1(\Omega) : {\bf v } = {\bf 0} \textup{ on } \Gamma \right\} \nonumber \\
L^2_0(\Omega) := \left\{ q \in L^2(\Omega) : \int_{\Omega}q d\Omega = 0 \right\} \nonumber
\end{align}
We also assume that $\sigma, \nu \in \mathbb{R}_+$, ${\bf a} \in {\bf H}^1_0(\Omega)$, and ${\bf f} \in {\bf L}^2(\Omega)$, and we assume that the advection velocity is divergence-free, that is, $\nabla \cdot {\bf a} \equiv 0$.  With these assumptions in hand, the weak form of the generalized Stokes or Oseen problem is stated as follows:  Find ${\bf u} \in {\bf H}^1_0(\Omega)$ and $p \in L^2_0(\Omega)$ such that:
\begin{equation}
a({\bf v},{\bf u}) - b({\bf v},p) + b({\bf u},q) = \ell({\bf v})
\end{equation}
for all ${\bf v} \in {\bf H}^1_0(\Omega)$ and $q \in L^2_0(\Omega)$ where:
\begin{equation}
a({\bf v},{\bf u}) := \left\{ \begin{array}{cl}
\displaystyle \int_\Omega \sigma {\bf v} \cdot {\bf u} \ d \Omega + \int_\Omega \nu  \nabla {\bf v} : \nabla{\bf u} \ d \Omega & \text{generalized Stokes}\\
\displaystyle \int_\Omega \sigma  {\bf v} \cdot {\bf u} \ d \Omega + \int_\Omega {\bf v} \cdot \left( {\bf a} \cdot \nabla {\bf u} \right) \ d \Omega + \int_\Omega \nu  \nabla {\bf v} : \nabla{\bf u} \ d \Omega & \text{generalized Oseen}\\
\end{array} \right. \nonumber
\end{equation}
\[
b({\bf v},p) := \displaystyle \int_\Omega \left( \nabla \cdot {\bf v} \right) p  \ d \Omega
\]
\[
\ell({\bf v}) := \displaystyle \int_\Omega {\bf v} \cdot {\bf f} d\Omega
\]

\subsection{Mixed Galerkin Approximation of the Generalized Stokes and Oseen Problems}

To discretize in space using a mixed Galerkin formulation, we first must specify finite-dimensional approximation spaces for the velocity and pressure fields.  We denote these spaces as ${\bf V}_h \subset {\bf H}^1_0(\Omega)$ and $\text{Q}_h \subset L^2_0(\Omega)$ respectively, but we defer the discussion of suitable approximation spaces to Section \ref{sec:DivConf}.  With approximation spaces defined, the mixed Galerkin formulation of the generalized Stokes or Oseen problem is stated as follows: Find ${\bf u}_h \in {\bf V}_h$ and $p_h \in \text{Q}_h$ such that
\begin{equation}
a({\bf v}_h,{\bf u}_h) - b({\bf v}_h,p_h) + b({\bf u}_h,q_h) = \ell({\bf v}_h)
\label{eqn:GalApp}
\end{equation}
for all ${\bf v}_h \in {\bf V}_h$ and $q \in \text{Q}_h$.  It should be noted that the velocity and pressure approximation spaces may not be arbitrarily selected.  Instead, they should be chosen such that the Babu\v{s}ka-Brezzi inf-sup condition is satisfied \cite{Babuska71,Brezzi74}.  We later select isogeometric divergence-conforming discretizations for spatial discretization which indeed satisfy such a condition.

\subsection{Weak Enforcement of No-Slip Boundary Conditions}

The no-slip boundary condition, ${\bf u} \times {\bf n} = {\bf 0}$ where ${\bf n}$ is the outward facing normal to $\Gamma$, leads to the formation of boundary layers for wall-bounded flows.  High mesh resolution is required near boundary layers to accurately represent associated sharp layers, so when the no-slip condition is strongly enforced in a mixed Galerkin formulation, inaccurate flow field approximations are obtained for insufficiently-resolved boundary layer meshes. It has recently been shown that superior results can be achieved by imposing the no-penetration boundary condition strongly and the no-slip boundary condition weakly using a combination of upwinding and Nitsche's method \cite{bazilevs07_1, bazilevs07_2}. With such an approach, we first specify finite-dimensional velocity and pressure approximation spaces as before, but we only require that the corresponding discrete velocity fields satisfy ${\bf v} \cdot {\bf n} = 0$. That is, we specify ${\bf V}_h \subset {\bf H}^1_n(\Omega) = \left\{ {\bf v} \in {\bf H}^1(\Omega): {\bf v} \cdot {\bf n} = 0 \textup{ on } \Gamma \right\}$ and $\text{Q}_h \subset L^2_0(\Omega)$.  The corresponding formulation for the generalized Stokes or Oseen problem is then stated as: Find ${\bf u}_h \in {\bf V}_h$ and $p_h \in \text{Q}_h$ such that:
\begin{equation}
a_h({\bf v}_h,{\bf u}_h) - b({\bf v}_h,p_h) + b({\bf u}_h,q_h) = \ell({\bf v}_h)
\label{eqn:WeakEnforce}
\end{equation}
for all ${\bf v}_h \in {\bf V}_h$ and $q_h \in \text{Q}_h$ where:
\begin{equation*}
a_h({\bf v}_h,{\bf u}_h) := a({\bf v}_h,{\bf u}_h) - \int_\Gamma \nu {\bf v}_h \cdot \nabla_{\bf n} {\bf u}_h \ d \Gamma - \int_\Gamma \nu \nabla_{\bf n} {\bf v}_h \cdot {\bf u}_h \ d \Gamma + \int_\Gamma \frac{C_I \nu}{h} {\bf v}_h \cdot {\bf u}_h \ d \Gamma
\end{equation*}
Above, $h$ is the wall-normal element mesh size and $C_I$ is a positive constant that must be chosen sufficiently large to ensure coercivity of the bilinear form $a_h(\cdot,\cdot)$.  Appropriate values for the constant $C_I$ can be obtained by solving element-wise eigenvalue problems or by appealing to analytical upper bounds for the trace inequality \cite{Evans13_4}.  We choose to weakly enforce no-slip boundary conditions throughout the remainder of this work.  This not only leads to more accurate numerical results, but it also ensures proper solution behavior in the limit of zero viscosity \cite{Evans13_2,Evans13_3}.

\subsection{The Matrix Problem}

The formulation given by \eqref{eqn:WeakEnforce} yields a linear matrix system when the discrete velocity and pressure spaces are provided basis functions. Let $\left\{ {\bf N}^v_i \right\}_{i=1}^{n_v}$ denote a set of vector basis functions for ${\bf V}_h$ where $n_v = \text{dim}\left({\bf V}_h\right)$, and let $\left\{ N^q_i \right\}_{i=1}^{n_q}$ denote a set of scalar basis functions for $\text{Q}_h$ where $n_q = \text{dim}\left(\text{Q}_h\right)$. Then the resulting matrix system takes the form:

\begin{equation}
\left[ \begin{array}{cc}
    {\bf A} & -{\bf B} \\
    {\bf B}^T & {\bf 0} 
\end{array} \right] \left( \begin{array}{c}
{\bf u}\\
{\bf p} \end{array} \right) = \left( \begin{array}{c}
{\bf f}\\
{\bf 0} \end{array} \right)
\label{eqn:BlockMatSys}
\end{equation}
where:
\begin{align}
[{\bf A}]_{ij} &:= a_h({\bf N}^v_i,{\bf N}^v_j) \nonumber \\
[{\bf B}]_{ij} &:= b({\bf N}^v_i, N_j^q) \nonumber \\
[{\bf f}]_i &:= \ell({\bf N}^v_i) \nonumber
\end{align}
Moreover, this matrix system can be written concisely as:
\begin{equation}
{\bf K} {\bf U} = {\bf F}
\end{equation}
where the matrix ${\bf K}$ has the block structure in \eqref{eqn:BlockMatSys} and the vectors ${\bf U}$ and ${\bf F}$ are vectors representing the group variables in \eqref{eqn:BlockMatSys}.

\section{The Stokes Complex and Isogeometric Compatible Discretizations}
\label{sec:DivConf}

It remains to specify suitable velocity and pressure approximation spaces for the generalized Stokes and Brinkman problems.  In this section, we present a particular selection of velocity and pressure approximation spaces which is not only inf-sup stable but also yields pointwise divergence-free discrete velocity fields.  Before doing so, however, we first introduce the so-called Stokes complex which succinctly captures the fundamental theorem of calculus and expresses the differential relationships between potential, velocity, and pressure fields.

\subsection{The Stokes Complex}
The Stokes complex \cite{Evans11,Falk13} is a cochain complex of the form:
\begin{align}
\begin{CD}
0 @>>> \Phi @>\vec{\nabla}>> \bm{\Psi} @>\vec{\nabla}\times>> {\bf V} @>\vec{\nabla}\cdot>> \text{Q} @>>> 0
\end{CD}
\label{eqn:StokesComplex}
\end{align}
in the three-dimensional setting where:
\begin{align}
\Phi &:= H^1_0(\Omega) &
\bm{\Psi} &:= \left\{ \bm{\psi} \in {\bf L}^2(\Omega): \vec{\nabla}\times \bm{\psi} \in {\bf H}^1(\Omega) \textup{ and } \bm{\psi} \times {\bf n} = {\bf 0} \textup{ on } \Gamma \right\} \nonumber \\
{\bf V} &:= {\bf H}^1_n(\Omega) &
\text{Q} &:= L^2_0(\Omega) \nonumber
\end{align}
are infinite-dimensional spaces of scalar potential fields, vector potential fields, velocity fields, and pressure fields. The Stokes complex is a smoothed version of the classical $L^2$ de Rham complex, and when the domain $\Omega \subset \mathbb{R}^3$ is simply connected with simply connected boundary, the Stokes complex is exact.  This means that every pressure field may be represented as the divergence of a velocity field, every divergence-free velocity field may be represented as the curl of a vector potential field, and every curl-free vector potential field may be represented as the gradient of a scalar potential field.  An analogous two-dimensional Stokes complex also exists, though for brevity, the interested reader is referred to \cite{Evans11} for more details.

It has been shown in previous works that the Stokes complex endows the generalized Stokes and Oseen problems with important underlying topological structure.  In particular, the infinite-dimensional inf-sup condition may be derived from the complex \cite{Evans11}.  As such, there is impetus for developing finite-dimensional approximations of the Stokes complex.  Such discrete complexes are referred to as discrete Stokes complexes, and when these complexes are endowed with special commuting projection operators, they form the following commuting diagram with the Stokes complex:
\begin{align}
\begin{CD}
0 @>>> \Phi @>\vec{\nabla}>> \bm{\Psi} @>\vec{\nabla}\times>> {\bf V} @>\vec{\nabla}\cdot>> \text{Q} @>>> 0\\
@. @VV \Pi_\phi V @VV \Pi_\psi V @VV \Pi_v V @VV \Pi_q V \\
0 @>>> \Phi_h @>\vec{\nabla}>> \bm{\Psi}_h @>\vec{\nabla}\times>> {\bf V}_h @>\vec{\nabla}\cdot>> \text{Q}_h @>>> 0
\end{CD}
\label{eqn:StokesComplexCD}
\end{align}
where $\Phi_h$, $\bm{\Psi}_h$, ${\bf V}_h$, and $\text{Q}_h$ are discrete scalar potential, vector potential, velocity, and pressure spaces and $\Pi_\phi : \Phi \rightarrow \Phi_h$, $\Pi_\psi: \bm{\Psi} \rightarrow \bm{\Psi}_h$, $\Pi_v : {\bf V} \rightarrow {\bf V}_h$, and $\Pi_q : \text{Q} \rightarrow \text{Q}_h$ are the aforementioned commuting projection operators.  Remarkably, when ${\bf V}_h$ and $\text{Q}_h$ are selected as velocity and pressure approximation spaces in a mixed Galerkin formulation of the generalized Stokes or Oseen problem, the resulting approximation scheme is \textit{inf-sup stable and free of spurious oscillations} and the \textit{returned discrete velocity solution will be pointwise divergence-free} \cite{Evans13_1,Evans13_2}.  Both of these properties are a direct consequence of the commuting diagram above, and for the sake of completeness, we prove the second property below.

\begin{lemma}
Assume that the discrete velocity and pressure spaces ${\bf V}_h$ and $\text{Q}_h$ are associated with a discrete complex which commutes with the Stokes complex.  Suppose ${\bf v}_h \in {\bf V}_h$ satisfies $b({\bf v}_h,q_h) = 0$ for every $q_h \in \text{Q}_h$.  Then $\nabla \cdot {\bf v}_h = 0$ pointwise.
\end{lemma}

\begin{proof}
Let $q_h = \nabla \cdot {\bf v}_h$.  Then $\| \nabla \cdot {\bf v}_h \|^2_{L^2(\Omega)} = b({\bf v}_h,q_h) = 0$ and the desired result follows.
\end{proof}

While we have demonstrated the benefit of using velocity and pressure spaces coming from a discrete Stokes complex, we have not yet described how to arrive at such spaces.  In this paper, we turn to the use of so-called isogeometric compatible B-spline discretizations which are the focus of the next two subsections.

\subsection{Univariate and Multivariate B-splines}

The basic building blocks of isogeometric compatible B-spline discretizations, like any isogeometric analysis technology, are B-splines \cite{Cottrell09}.  B-splines are piecewise polynomial functions, but unlike $C^0$-continuous finite elements, B-splines may exhibit high levels of continuity.  Univariate B-splines are constructed by first specifying a polynomial degree $p$\footnote{The notation $p$ is used for both the pressure field as well as the polynomial degree.  Thus, the reader should discern what term $p$ refers to in various portions of the paper by context.}, a number of basis functions $n$, and an open knot-vector $\Xi = \left\{\xi_0,\xi_1,\ldots,\xi_{n+p+1}\right\}$, a non-decreasing vector of knots $\xi_i$ such that the first and last knot are repeated $p+1$ times. We assume without loss of generality that the first and last knot are 0 and 1 respectively such that the domain of the knot vector is $(0,1)$.  With a knot vector in hand, univariate B-spline basis functions are defined recursively through the Cox-deBoor formula:
\begin{align}
\begin{array}{rl}
	\hat{N}_{i,p}(\xi) &:= \frac{\xi - \xi_i}{\xi_{i+p}-\xi_i}\hat{N}_{i,p-1}(\xi) + \frac{\xi_{i+p+1}-\xi}{\xi_{i+p+1}-\xi_{i+1}}\hat{N}_{i+1,p-1}(\xi) \hspace{5pt} \textup{ for } p > 0\\ \\
	\hat{N}_{i,0}(\xi) &:= \left\{\begin{array}{rl}
			1& \xi_i \le \xi < \xi_{i+1}\\
		0& \text{elsewhere} \end{array} \right.
		\end{array}
		\nonumber
\end{align}
Figure \ref{fig:bases} shows example sets of unvariate B-spline basis functions.  We can alternatively define B-splines not from the knot vector itself, but instead a vector of unique knot values $\bm{\zeta} = \left\{\zeta_1, \zeta_2, \ldots, \zeta_{n_k}\right\}$ and a regularity vector $\bm{\alpha} = \left\{\alpha_1, \alpha_2, \ldots, \alpha_{n_k}\right\}$ such that the B-splines have $\alpha_j$ continuous derivatives across $\zeta_j$.  By construction, $\alpha_1 = \alpha_{n_k} = -1$.  We will later employ the convention $\bm{\alpha} - 1 = \left\{ -1, \alpha_2 - 1, \ldots, \alpha_{n_k-1}-1, -1 \right\}$.

\begin{figure}[b!]
\centering
\includegraphics[scale=0.9]{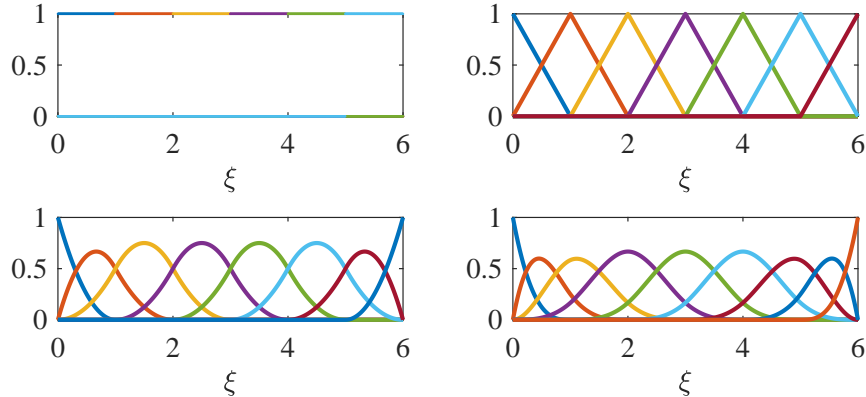}
\captionsetup{format = hang}
\caption{Maximal continuity univariate B-spline basis functions of varying polynomial degree associated with a vector of unique knot values $\bm{\zeta} = \left\{0,1,2,3,4,5,6\right\}$: $p = 0$ (upper left), $p = 1$ (upper right), $p = 2$ (lower left), and $p = 3$ (lower right).}
\label{fig:bases}
\end{figure}

Given a set of knot-vectors and polynomial degrees, multivariate B-spline basis functions are obtained through a tensor-product of unvariate B-spline basis functions:
\[
\hat{N}_{{\bf i},{\bf p}}(\bm{\xi}) := \prod_{k=1}^d \hat{N}_{i_k,p_k}(\xi_k)
\]
where ${\bf i} = (i_1, i_2, \ldots, i_d)$ and ${\bf p} = (p_1, p_2, \ldots, p_d)$
We denote the corresponding space of multidimensional B-splines over the parametric domain $\hat{\Omega} = (0,1)^d$ as:
\[
S^{p_1,p_2,\ldots,p_d}_{\bm{\alpha}_1,\bm{\alpha}_2,\ldots,\bm{\alpha}_d}\left(\mathcal{M}_h\right) := \left\{ f : \hat{\Omega} \rightarrow \mathbb{R} \ \Big| \ f(\bm{\xi}) = \sum_{\bf i} a_{\bf i} \hat{N}_{{\bf i},{\bf p}} (\bm{\xi}) \right\},
\]
where $\bm{\alpha}_j$ is the regularity vector associated with the $j^{\text{th}}$ direction where $j = 1, \ldots, d$ and $\bm{\mathcal{M}}_h$ is the parametric mesh defined by the vectors of unit knot values in each parameteric direction.  Note that the space is fully characterized by the polynomial degrees, regularity vectors, and parametric mesh as indicated by the notation.  For ease of notation, however, we drop the dependence on the parameteric mesh and instead use $S^{p_1,p_2,\ldots,p_d}_{\bm{\alpha}_1,\bm{\alpha}_2,\ldots,\bm{\alpha}_d} = S^{p_1,p_2,\ldots,p_d}_{\bm{\alpha}_1,\bm{\alpha}_2,\ldots,\bm{\alpha}_d}\left(\mathcal{M}_h\right)$ in what follows.

\subsection{Isogeometric Compatible B-splines}

We are now in a position to define isogeometric compatible B-splines.  Their definition is made possible through the observation that the derivative of univariate B-splines of degree $p$ are univariate B-splines of degree $p - 1$.  Since multivariate B-splines are tensor-products of univariate B-splines, the aforementioned property naturally generalizes to higher dimension, allowing us to build a discrete Stokes complex of B-spline spaces \cite{Evans11,Buffa11_diff}.  We first define such a discrete Stokes complex in the parametric domain $\hat{\Omega} = (0,1)^d$ for both $d = 2$ and $d = 3$ before constructing a discrete Stokes complex in the physical domain of interest using a set of structure-preserving push-forward/pull-back operators.

In the two-dimensional setting, we define the following B-spline spaces over the unit square:
\begin{equation}
\begin{aligned}
\hat{\Psi}_h &:= \left\{ \hat{\psi}_h \in S_{\bm{\alpha}_1,\bm{\alpha}_2}^{p_1,p_2}: \hat{\psi}_h = 0 \textup{ on } \hat{\Gamma} \right\}\\
\hat{\bf V}_h &:= \left\{ \hat{{\bf v}}_h \in S_{\bm{\alpha}_1,\bm{\alpha}_2-1}^{p_1,p_2-1} \times S_{\bm{\alpha}_1-1,\bm{\alpha}_2}^{p_1-1,p_2}: \hat{{\bf v}}_h \cdot {\bf n} = 0 \textup{ on } \hat{\Gamma} \right\}\\
\hat{\text{Q}}_h &:= \left\{ \hat{q}_h \in S_{\bm{\alpha}_1-1,\bm{\alpha}_2-1}^{p_1-1,p_2-1}: \int_{\hat{\Omega}} \hat{q}_h d\hat{\Omega} = 0 \right\}
\label{eqn:2DdivConfSpace} \nonumber
\end{aligned}
\end{equation}
where $\hat{\Psi}_h$ is the B-spline space of streamfunctions, $\hat{\bf V}_h$ is the B-spline space of flow velocities, and $\hat{\text{Q}}_h$ is the B-spline space of pressures. These discrete spaces are endowed with B-spline basis functions $\{ \hat{N}^\psi_i \}_{i=1}^{n_\psi}$, $\{ \hat{\bf N}^v_i \}_{i=1}^{n_v}$, and $\{ \hat{N}^p_i \}_{i=1}^{n_q}$, respectively, where $n_\psi$ is the number of streamfunction basis functions, $n_v$ is the number of velocity basis functions, and $n_q$ is the number of pressure basis functions, all of which can be inferred from the chosen polynomial degrees and knot vectors. One can readily show that these spaces form the following discrete Stokes complex:
\begin{align}
\begin{CD}
0 @>>> \hat{\Psi}_h @>\vec{\nabla}^\perp>> \hat{\bf V}_h @>\vec{\nabla}\cdot>> \hat{\text{Q}}_h @>>> 0
\end{CD}
\label{eqn:2DStructurePreserveBSplineCD}
\end{align}
and provided the functions in the B-spline pressure space are at least $C^0$-continuous, there exist a set of commuting projection operators that make the above discrete complex commute with the Stokes complex.  Thus, we refer to the spaces $\hat{\Psi}_h$, $\hat{{\bf V}}_h$, and $\hat{\text{Q}}_h$ as compatible B-spline spaces.  As mentioned previously, if we select $\hat{{\bf V}}_h$ and $\hat{\text{Q}}_h$ as velocity and pressure approximation spaces in a mixed Galerkin formulation of the generalized Stokes or Oseen problems, then the resulting scheme yields a pointwise divergence-free velocity field.  The degrees of freedom associated with compatible B-splines are associated with the geometrical entries of the underlying control mesh.  This is graphically illustrated in Figure \ref{fig:control_mesh} which shows that streamfunction degrees of freedom are associated with control points, velocity degrees of freedom are associated with (and aligned normal to) control edges, and pressure degrees of freedom are associated with control cells.  Each degree of freedom corresponds to a particular basis function, and to visualize these basis functions, we have selected four degrees of freedom in Figure \ref{fig:control_mesh} and visualized the respective basis functions in Figure \ref{fig:StructPreserveBSpline}\footnote{Note that the pressure basis function we have highlighted does not have zero average over the parametric domain.  In practice, we enforce this constraint using a Lagrange multiplier rather than to the individual pressure basis functions.}.

\begin{figure}[t!]
\centering
\includegraphics[scale=0.65]{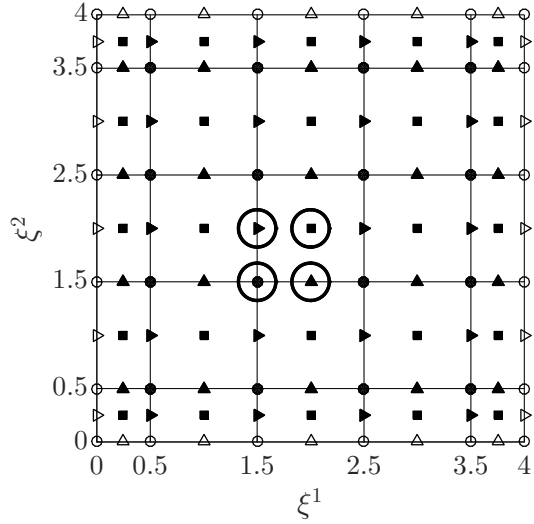}
\captionsetup{format = hang}
\caption{Control mesh and degrees of freedom for maximal continuity compatible B-splines of degree $p_1 = p_2 = 2$ associated with vectors of unique knot values $\bm{\zeta}_1 = \bm{\zeta}_2 = \left\{0, 1, 2, 3, 4\right\}$. Filled circles (\ding{108}) denote streamfunction degrees of freedom, triangles ($\blacktriangleright$, $\blacktriangle$) denote velocity degrees of freedom, and filled squares (\ding{110}) denote pressure degrees of freedom. Hollow markers indicate degrees of freedom associated with boundary conditions.}
\label{fig:control_mesh}
\end{figure}

\begin{figure}[t!]
\centering
\begin{subfigure}{.45\textwidth}
	\includegraphics{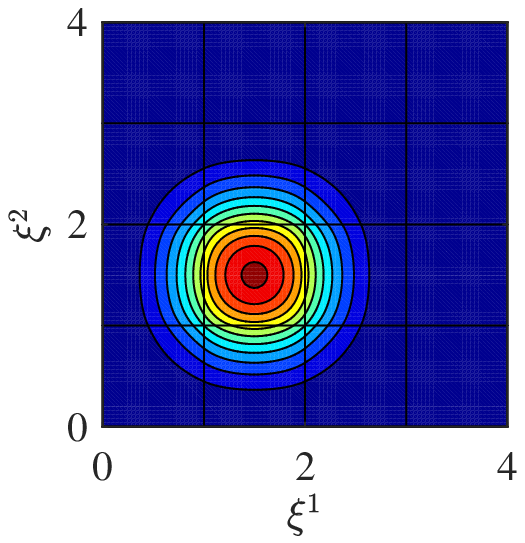}
\end{subfigure}
\begin{subfigure}{.45\textwidth}
	\includegraphics{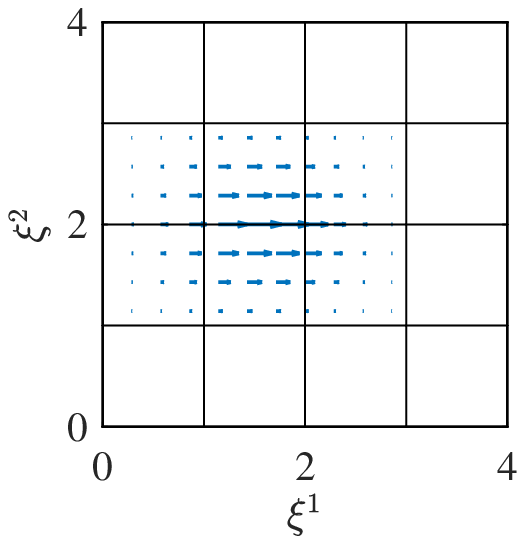}
\end{subfigure}\\
\begin{subfigure}{.45\textwidth}
	\includegraphics{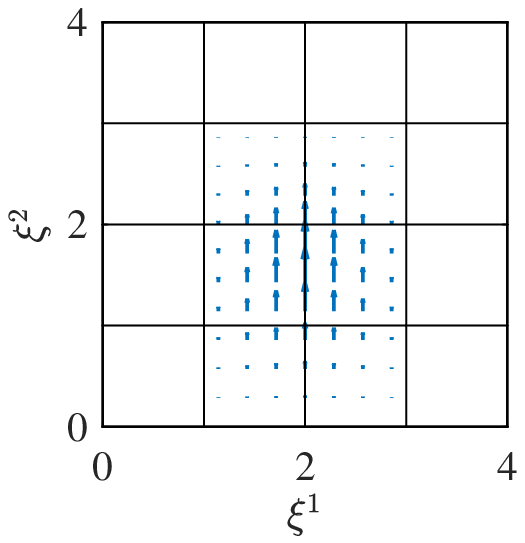}
\end{subfigure}
\begin{subfigure}{.45\textwidth}
	\includegraphics{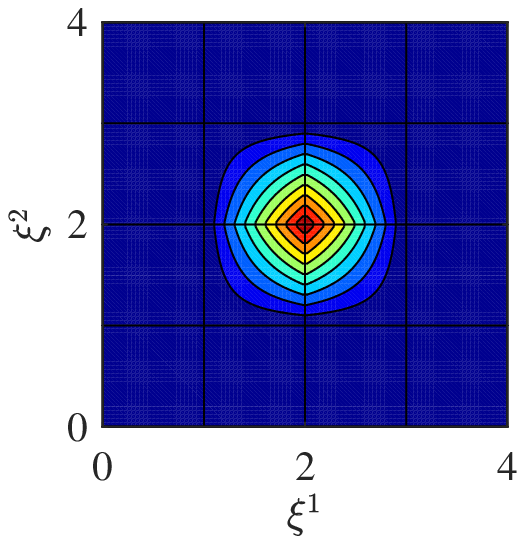}
\end{subfigure}
\caption{Streamfunction (upper left), velocity (upper right and lower left), and pressure (lower right) basis functions associated with the circled degrees of freedom in Figure \ref{fig:control_mesh}.}
\label{fig:StructPreserveBSpline}
\end{figure}

In the three-dimensional setting, we define the following B-spline spaces over the unit cube:
\begin{equation}
\begin{aligned}
\hat{\Phi}_h &:= \left\{ \hat{\phi}_h \in S_{\bm{\alpha}_1,\bm{\alpha}_2,\bm{\alpha}_3}^{p_1,p_2,p_3}: \hat{\phi}_h = 0 \textup{ on } \hat{\Gamma} \right\}\\
\hat{\bf \Psi}_h &:= \left\{ \hat{\bm{\psi}}_h \in S_{\bm{\alpha}_1-1,\bm{\alpha}_2,\bm{\alpha}_3}^{p_1-1,p_2,p_3} \times S_{\bm{\alpha}_1,\bm{\alpha}_2-1,\bm{\alpha}_3}^{p_1,p_2-1,p_3} \times S_{\bm{\alpha}_1,\bm{\alpha}_2,\bm{\alpha}_3-1}^{p_1,p_2,p_3-1} : \hat{\bm{\psi}}_h \times {\bf n} = {\bf 0} \textup{ on } \hat{\Gamma} \right\}\\
\hat{\bf V}_h &:= \left\{ \hat{\bf v}_h \in  S_{\bm{\alpha}_1,\bm{\alpha}_2-1,\bm{\alpha}_3-1}^{p_1,p_2-1,p_3-1} \times S_{\bm{\alpha}_1-1,\bm{\alpha}_2,\bm{\alpha}_3-1}^{p_1-1,p_2,p_3-1} \times S_{\bm{\alpha}_1-1,\bm{\alpha}_2-1,\bm{\alpha}_3}^{p_1-1,p_2-1,p_3} : \hat{\bf v}_h \cdot {\bf n} = 0 \textup{ on } \hat{\Gamma} \right\}\\
\hat{\text{Q}}_h &:= \left\{ \hat{q}_h \in S_{\bm{\alpha}_1-1,\bm{\alpha}_2-1,\bm{\alpha}_3-1}^{p_1-1,p_2-1,p_3-1} : \int_{\hat{\Omega}} \hat{q}_h d\hat{\Omega} = 0 \right\}
\label{eqn:3DdivConfSpace} \nonumber
\end{aligned}
\end{equation}
where $\hat{\Phi}_h$ is the B-spline space of scalar potentials, $\hat{\bf \Psi}_h$ is the B-spline space of vector potentials, $\hat{\bf V}_h$ is the B-spline space of flow velocities, and $\hat{\text{Q}}_h$ is the B-spline space of pressures. These discrete spaces are endowed with the basis functions $\{ \hat{N}^\phi_i \}_{i=1}^{n_\phi}$, $\{ \hat{\bf N}^\psi_i \}_{i=1}^{n_\psi}$, $\{ \hat{\bf N}^v_i \}_{i=1}^{n_v}$, and $\{ \hat{N}^p_i \}_{i=1}^{n_q}$, respectively, where $n_\phi$ is the number of scalar potential basis functions, $n_\psi$ is the number of vector potential basis functions, $n_v$ is the number of velocity basis functions, and $n_q$ is the number of pressure basis functions, all of which can be inferred from the chosen polynomial degrees and knot vectors. Once again, one can show that the above spaces form the following discrete Stokes complex:
\begin{align}
\begin{CD}
0 @>>> \hat{\Phi}_h @>\vec{\nabla}>> \hat{\bf \Psi}_h @>\vec{\nabla}\times>> \hat{\bf V}_h @>\vec{\nabla}\cdot>> \hat{\text{Q}}_h @>>> 0
\end{CD}
\label{eqn:3DStructurePreserveBSplineCD}
\end{align}
and provided the functions in the B-spline pressure space are at least $C^0$-continuous, there exist a set of commuting projection operators that make the above discrete complex commute with the Stokes complex.

Heretofore, we have discussed how to construct compatible B-splines in the parametric domain.  To define compatible B-splines in the physical domain $\Omega$, we need to first define a piece-wise smooth bijective mapping ${\bf F} : \hat{\Omega} \rightarrow \Omega$.  This mapping can be defined using Non-Uniform Rational B-splines (NURBS), for instance, as is commonly done in the isogeometric analysis community \cite{Cottrell09}.  With this mapping in hand, we define two-dimensional compatible B-spline spaces in the physical domain via the relations:
\begin{equation}
\begin{aligned}
\Psi_h &:= \left\{ \psi_h \in \Psi: \psi_h \circ {\bf F} \in \hat{\Psi}_h \right\} \\
{\bf V}_h &:= \left\{ {\bf v}_h \in {\bf V}: \text{det}\left({\bf J}\right){\bf J}^{-1} {\bf v}_h \circ {\bf F} \in \hat{{\bf V}}_h \right\} \\
\text{Q}_h &:= \left\{ q_h \in \text{Q}: \text{det}\left({\bf J}\right) q_h \circ {\bf F} \in \hat{\text{Q}}_h \right\}
\label{eqn:physical_2D} \nonumber
\end{aligned}
\end{equation}
and three-dimensional compatible B-spline spaces via the relations:
\begin{equation}
\begin{aligned}
\Phi_h &:= \left\{ \phi_h \in \Phi: \phi_h \circ {\bf F} \in \hat{\Phi}_h \right\} \\
\bm{\Psi}_h &:= \left\{ \bm{\psi}_h \in \bm{\Psi}: {\bf J}^{-T}\bm{\psi}_h \circ {\bf F} \in \hat{\bm{\Psi}}_h \right\} \\
{\bf V}_h &:= \left\{ {\bf v}_h \in {\bf V}: \text{det}\left({\bf J}\right){\bf J}^{-1} {\bf v}_h \circ {\bf F} \in \hat{{\bf V}}_h \right\} \\
\text{Q}_h &:= \left\{ q_h \in \text{Q}: \text{det}\left({\bf J}\right) q_h \circ {\bf F} \in \hat{\text{Q}}_h \right\}
\label{eqn:physical_3D} \nonumber
\end{aligned}
\end{equation}
where  ${\bf J} = \partial_{\bm{\xi}} {\bf F}$ is the Jacobian of the parametric mapping. Corresponding basis functions in the physical domain are defined via push-forwards of the basis functions in the parametric domain, and we denote the discrete velocity basis functions as $\{{\bf N}^v_i \}_{i=1}^{n_v}$ and the basis functions for other quantities in analogous fashion.  It is easily shown that the compatible B-spline spaces in the physical domain also comprise a discrete complex which commutes with the Stokes complex.  The compatible B-splines in the physical domain are referred to as isogeometric compatible B-splines as they are built from B-splines, the basis building blocks of geometric modeling, and they are defined on the exact geometry of the problem of interest.

\subsection{B-spline Refinement}
\label{subsec:refinement}

\begin{figure}[b!]
\centering
\includegraphics[scale=0.9]{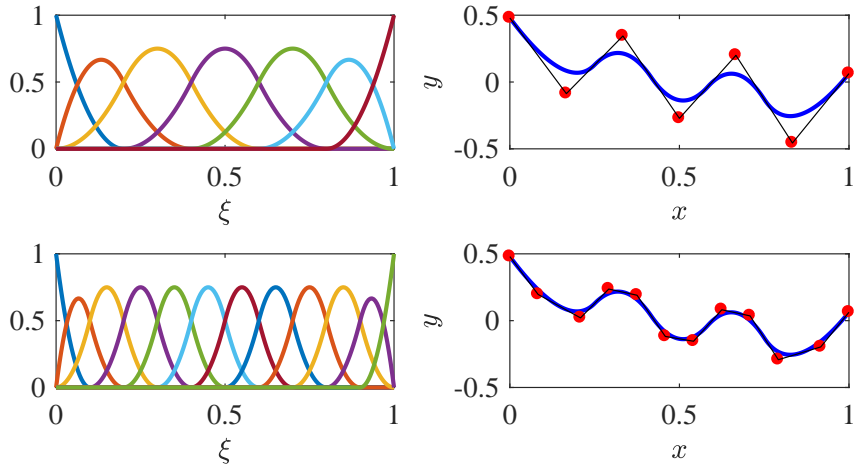}
\captionsetup{format = hang}
\caption{The action of knot insertion for univariate quadratic B-splines.  Top Left: The original quadratic B-spline basis with $\Xi=(0, 0, 0, .2, .4, .6, .8, 1, 1, 1)$.  Top Right: A B-spline function expressed in terms of the original quadratic B-spline basis.  Bottom Left: The refined quadratic B-spline basis with $\tilde{\Xi} = (0, 0, 0, .1, .2, .3, .4, .5, .6, .7, .8, .9, 1, 1, 1)$ (bottom, left).  Bottom right: The same B-spline function as illustrated in the top right panel but expressed in terms of the refined quadratic B-spline basis.}
\label{fig:NURBSrefine}
\end{figure}

One more concept needs to be introduced before proceeding forward, namely the concept of B-spline refinement.  For a fixed set of polynomial degrees, B-spline refinement is carried out by a process referred to as knot insertion \cite{Goldman92}.  In the univariate setting, we start with a particular knot vector $\Xi$ and then insert a sequence of knots to arrive at a refined knot vector $\tilde{\Xi}$ such that $\Xi \subset \tilde{\Xi}$.  The B-spline basis functions associated with the original knot vector, denoted as $\left\{ \hat{N}_{i_p}(\xi) \right\}_{i=1}^n$, can be represented as linear combinations of the basis functions associated with the refined knot vector, denoted as $\left\{ \tilde{N}_{i,p}(\xi) \right\}_{i=1}^{\tilde{n}}$, using a transformation matrix ${\bf T}$.  This relationship is expressed mathematically as:
\begin{equation}
\label{eqn:knot_insertion}
\hat{N}_{i,p}(\xi) =  \sum_{j = 1}^{\tilde{n}} [{\bf T}]_{ij} \tilde{N}_{j,p}(\xi) \nonumber
\end{equation}
for $i = 1, \ldots, n$.  Consequently, if a B-spline function takes the form:
\begin{equation}
\hat{u}(\xi) = \sum_{i=1}^{n} \hat{u}_i \hat{N}_{i,p}(\xi)  \nonumber
\end{equation}
it can be alternately be represented as:
\begin{equation}
\hat{u}(\xi) = \sum_{j=1}^{\tilde{n}} \tilde{u}_j \tilde{N}_{j,p}(\xi)  \nonumber
\end{equation}
where:
\begin{equation}
\tilde{u}_j =  \sum_{i = 1}^{n} [{\bf T}]_{ij} \hat{u}_i  \nonumber
\end{equation}
for $j = 1, \ldots, \tilde{n}$.  Figure \ref{fig:NURBSrefine} depicts the action of knot insertion for univariate quadratic B-splines.   B-spline refinement in the multivariate setting (including the compatible B-spline setting) is carried out in a tensor-product fashion, and the transformation matrix ${\bf T}$ takes the same form in both the parametric domain and the physical domain.  There exist a variety of algorithms capable of performing knot insertion \cite{Goldman92,Piegl12,Thomas15} which can be used to construct the transformation matrix ${\bf T}$, so we do not discuss this construction further in this paper.

\section{A Structure-Preserving Geometric Multigrid Methodology}
\label{sec:multigrid}

At last, we are ready to present our geometric multigrid methodology for isogeometric compatible discretizations of the generalized Stokes and Oseen problems.  We begin this section by reviewing the basics of the geometric multigrid approach as well as the required ingredients in the setting of an isogeometric compatible discretization.  Then, we introduce the Schwarz-style smoothers which our methodology leans upon.  We then show that our methodology preserves the divergence-free constraint on the velocity field and that it effectively ellipticizes the underlying system of interest.  We limit our discussion to the V-cycle algorithm, though our approach can also be applied within a W-cycle or Full Multigrid framework \cite{Briggs00}.

\subsection{Nested B-spline Stokes Complexes, Intergrid Transfer Operators, and the V-Cycle Algorithm}

Assume that we have a sequence of nested B-spline Stokes complexes that have been obtained through knot insertion.  We denote the discrete velocity and pressure spaces associated with this sequence as $\left\{ {\bf V}_{\ell} \right\}_{\ell = 0}^{n_\ell}$ and $\left\{ \text{Q}_{\ell} \right\}_{\ell = 0}^{n_\ell}$ respectively where $n_\ell$ is the number of levels, and we note that:
\begin{align}
{\bf V}_{0} \subset {\bf V}_{1} \subset \ldots \subset {\bf V}_{n_\ell} \nonumber \\
\text{Q}_{0} \subset \text{Q}_{1} \subset \ldots \subset \text{Q}_{n_\ell} \nonumber
\end{align}
and:
\begin{equation}
\nabla \cdot {\bf V}_{\ell} = \text{Q}_{\ell} \nonumber
\end{equation}
for each $\ell = 0, \ldots, n_\ell$.  Level $\ell = 0$ corresponds to the coarsest mesh while level $\ell = n_\ell$ corresponds to the finest mesh.  The action of knot insertion not only allows for B-spline refinement, but it also provides the intergrid transfer operators associated with a geometric multigrid method.  Namely, we can build prolongation operators:
\[
P^v_{\ell}: {\bf V}_{\ell} \rightarrow {\bf V}_{\ell+1} \hspace{15pt} \textup{and} \hspace{15pt} P^q_{\ell}: \text{Q}_{\ell} \rightarrow \text{Q}_{\ell+1}
\]
for $\ell = 0, \ldots, n_\ell - 1$ using the construction provided in Subsection \ref{subsec:refinement}.  We encode the action of these prolongation operators in the matrices ${\bf P}^v_\ell$ and ${\bf P}^q_\ell$ such that the following refinement operations hold:
\begin{align}
{\bf N}^v_{i,\ell}(\bm{\xi}) &= \sum_{j} [{\bf P}^v_\ell]_{ji} {\bf N}^v_{j,\ell+1}(\bm{\xi}) \nonumber \\
N^q_{i,\ell}(\bm{\xi}) &= \sum_{j} [{\bf P}^q_\ell]_{ji} N^q_{j,\ell+1}(\bm{\xi}) \nonumber
\end{align}
for $\ell = 0, \ldots, n_\ell - 1$ where $\left\{ {\bf N}^v_{i,\ell} \right\}_{i=1}^{n_{v,\ell}}$ and $\left\{ N^q_{i,\ell} \right\}_{i=1}^{n_{q,\ell}}$ denote the velocity and pressure B-spline basis functions associated with level $\ell$.  Moreover, the degrees of freedom associated with pressure and velocity fields on the $\ell^{\textup{th}}$ level can be transferred to the the $(\ell+1)^{\textup{st}}$ level via the expressions:
\begin{align}
{\bf u}_{\ell+1} &= {\bf P}^v_\ell {\bf u}_{\ell} \nonumber \\
{\bf p}_{\ell+1} &= {\bf P}^q_\ell {\bf p}_{\ell} \nonumber
\end{align}
As is standard with a Galerkin formulation, restriction operators are constructed as the adjoint or transpose of the prolongation operators, namely $R^v_{\ell+1} = \left(P^v_{\ell}\right)^{*}$, $R^q_{\ell+1} = \left(P^q_{\ell}\right)^{*}$, ${\bf R}^v_{\ell+1} = \left({\bf P}^v_{\ell}\right)^{T}$, and ${\bf R}^q_{\ell+1} = \left({\bf P}^q_{\ell}\right)^{T}$ for $\ell = 0, \ldots, n_\ell - 1$.  Finally, we define a prolongation matrix ${\bf P}_{\ell}$ for the full group variable such that:
\begin{align}
{\bf U}_{\ell+1} = \left[ \begin{array}{c}
{\bf u}_{\ell+1} \\
{\bf p}_{\ell+1}
\end{array} \right] = \left[ \begin{array}{cc}
{\bf P}^v_\ell & {\bf 0} \\
{\bf 0} & {\bf P}^q_\ell
\end{array} \right] \left[ \begin{array}{c}
{\bf u}_{\ell} \\
{\bf p}_{\ell}
\end{array} \right] =
{\bf P}_\ell {\bf U}_{\ell} \nonumber
\end{align}
for $\ell = 0, \ldots, n_\ell-1$.  The corresponding restriction matrix for level $\ell$ is given by ${\bf R}_{\ell+1} = \left({\bf P}_{\ell}\right)^{T}$.

We need a few more ingredients before stating the multigrid V-cycle algorithm for our discretization scheme.  First of all, we need to form the matrix system associated with the finest level, ${\bf K}{\bf U} = {\bf F}$.  We then form the system matrices associated with coarser levels via the relation ${\bf K}_{\ell} = {\bf R}_{\ell+1} {\bf K}_{\ell+1} {\bf P}_{\ell}$ for $\ell = 0, \ldots, n_{\ell}-1$ where ${\bf K}_{n_\ell} = {\bf K}$.  Second of all, we need to choose a smoother for each level $\ell$ which we encode in a smoothing matrix ${\bf S}_{\ell}$, and and we need to select a number of pre-smoothing steps $\nu_1$ and post-smoothing steps $\nu_2$.  Third of all, we need to choose a suitable initial guess ${\bf U}$ for the solution on the finest level.  Then, one V-cycle corresponds to a single call of the form $\text{MGV}(n_\ell,{\bf U},{\bf F})$ to the recursive function defined below \cite{Briggs00}.

\begin{algorithm}
\caption{Multigrid V-Cycle Algorithm}
\begin{algorithmic}[1]
\Function{MGV}{$\ell,{\bf U},{\bf F}$} 
 \If{$\ell = 0$}
   \State{${\bf U} = {\bf K}_{\ell}^{-1} {\bf F}$} \Comment{Exact System Solution}
 \Else
   \For{$i = 1$ to $\nu_1$}
        \State{${\bf U} \leftarrow {\bf U} + {\bf S}_{\ell}^{-1} \left( {\bf F} - {\bf K}_{\ell} {\bf U} \right)$} \Comment{Pre-Smoothing}
    \EndFor
    \State{${\bf G} = {\bf R}_{\ell} \left( {\bf F} - {\bf K}_{\ell} {\bf U} \right)$} \Comment{Restriction of Residual to Coarse Grid}
    \State{$\Delta {\bf U} = {\bf 0}$} \Comment{Coarse Grid Correction Initialization}
    \State \Call{MGV}{$\ell-1,\Delta {\bf U},{\bf G}$} \Comment{Coarse Grid Correction Evaluation}
    \State{${\bf U} \leftarrow  {\bf U} + {\bf P}_{\ell-1}\Delta {\bf U}$} \Comment{Update of Solution}
    \For{$i = 1$ to $\nu_2$}
        \State{${\bf U} \leftarrow {\bf U} + {\bf S}_{\ell}^{-1} \left( {\bf F} - {\bf K}_{\ell} {\bf U} \right)$} \Comment{Post-Smoothing}
    \EndFor
 \EndIf
\EndFunction
\end{algorithmic}
\end{algorithm}

\noindent Note that the solution ${\bf U}$ is updated within the algorithm stated above.  Hence, additional V-cycles simply correspond to additional calls of the form $\text{MGV}(n_\ell,{\bf U},{\bf F})$.

\subsection{Overlapping Schwarz Smoothers on Compatible Subdomains}

At this juncture, we have not yet determined what smoother to employ.  We turn to the use of overlapping Schwarz smoothers \cite{Dolean15} with specially chosen overlapping subdomains which respect the underlying topological structure of the generalized Stokes and Oseen problems \cite{Kanschat16}.  Namely, for each level $\ell$, we define a collection of subdomains $\{ \Omega_{i,\ell} \}_i$ where each individual subdomain is defined as the support of a discrete streamfunction basis function in the two-dimensional setting:
\[
\Omega_{i,\ell} := \text{supp}\left(N^{\psi}_{i,\ell}\right)
\]
and a discrete vector potential basis function in the three-dimensional setting:
\[
\Omega_{i,\ell} := \text{supp}\left({\bf N}^{\psi}_{i,\ell}\right)
\]
It is easily seen that the subdomains form a cover of the physical domain, that is:
\[
\overline{\Omega} = \bigcup_i \Omega_{i,\ell}
\]
For each subdomain, we define discrete velocity and pressure subspaces ${\bf V}_{i,\ell} \subset {\bf V}_\ell$ and $\text{Q}_{i,\ell} \subset \text{Q}_\ell$, respectively, as
\[
{\bf V}_{i,\ell} := \left\{ {\bf v}_h \in {\bf V}_{\ell}: \text{supp} \ {\bf v}_h \subseteq \Omega_{i,\ell} \right\} \hspace{15pt} \text{and} \hspace{15pt} \text{Q}_{i,\ell} := \left\{ q_h \in \text{Q}_{\ell} : \text{supp} \ q_h \subseteq \Omega_{i,\ell} \right\}
\]
In the two-dimensional setting, we define a discrete streamfunction subspace for each subdomain as:
\[
\Psi_{i,\ell} := \left\{ \psi_h \in \Psi_{\ell}: \text{supp} \ \psi_h \subseteq \Omega_{i,\ell} \right\}
\]
and in the three-dimensional setting, we define a discrete vector potential subspace for each subdomain as:
\[
\bm{\Psi}_{i,\ell} := \left\{ \bm{\psi}_h \in \bm{\Psi}_{\ell}: \text{supp} \ \bm{\psi}_h \subseteq \Omega_{i,\ell} \right\}
\]
The degrees of freedom associated with all of the aforementioned subspaces are illustrated in the two-dimensional case in Figure \ref{fig:subdomains} for two separate subdomains.

\begin{figure}[t!]
\centering
\includegraphics[scale=.75]{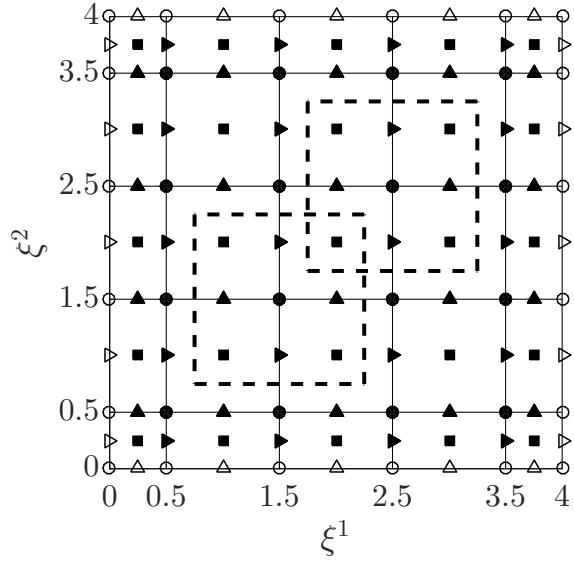}
\caption{The degrees of freedom associated with two example subdomains in the parametric domain $\hat{\Omega}$. Filled circles (\ding{108}) denote streamfunction degrees of freedom, triangles ($\blacktriangleright$, $\blacktriangle$) denote velocity degrees of freedom, and filled squares (\ding{110}) denote pressure degrees of freedom. Hollow markers indicate degrees of freedom associated with boundary conditions.}
\label{fig:subdomains}
\end{figure}

By construction, $\dim{\left(\Psi_{i,\ell}\right)} = 1$ in the two-dimensional setting and  $\dim{\left(\bm{\Psi}_{i,\ell}\right)} = 1$ in the three-dimensional setting. Moreover, $\dim{\left({\bf V}_{i,\ell}\right)} = 4$ and $\dim{\left(\text{Q}_{i,\ell}\right)} = 3$ in both the two- and three-dimensional settings\footnote{From Figure \ref{fig:subdomains}, it appears that $\dim{\left(\text{Q}_{i,\ell}\right)} = 4$.  However, the functions in $\text{Q}_{i,\ell}$ must satisfy a zero average constraint, so the dimension is one less than what is observed from the figure.}.  Thus, the subspaces associated with each subdomain form the following exact discrete Stokes complex in the two-dimensional setting:
\begin{align}
\begin{CD}
0 @>>> \Psi_{i,\ell} @>\vec{\nabla}^\perp>> {\bf V}_{i,\ell} @>\vec{\nabla}\cdot>> \text{Q}_{i,\ell} @>>> 0
\end{CD}
\label{eqn:Subdomain2DCD}
\end{align}
and the following exact discrete Stokes complex in the three-dimensional setting:
\begin{align}
\begin{CD}
0 @>>> {\bf \Psi}_{i,\ell} @>\vec{\nabla}\times>> {\bf V}_{i,\ell} @>\vec{\nabla}\cdot>> \text{Q}_{i,\ell} @>>> 0
\end{CD}
\label{eqn:Subdomain3DCD}
\end{align}
Thus, as previously suggested, our choice of subdomains indeed respects the underlying topological structure of the generalized Stokes and Oseen problems.

With our subdomains defined, we can now describe our choice of smoothers, namely additive and multiplicative Schwarz smoothers using our prescribed subdomains.  In this direction, let ${\bf E}^v_{i,\ell}$ and ${\bf E}^q_{i,\ell}$ denote the velocity and pressure subdomain restriction matrices for a given level $\ell$ and subdomain $i$ that take the full set of velocity and pressure degrees of freedom associated with level $\ell$ and map them to the set of pressure and velocity degrees of freedom associated with the subdomain $\Omega_{i,\ell}$.  Additionally, let:
\[
{\bf E}_{i,\ell}:= \left[ \begin{array}{cc}
{\bf E}^v_{i,\ell} & {\bf 0} \\
{\bf 0} & {\bf E}^q_{i,\ell}
\end{array} \right]
\]
denote the subdomain restriction matrix for a given level $\ell$ and subdomain $i$ for the full group variable.  Then, the action of the additive Schwarz smoother is defined through:
\begin{equation}
\label{eqn:ASM}
	{\bf S}^{-1}_\ell = \eta\left( \sum_i {\bf E}_{i,\ell}^T ({\bf E}_{i,\ell} {\bf K}_\ell {\bf E}_{i,\ell}^T)^{-1} {\bf E}_{i,\ell} \right)
\end{equation}
where $\eta \in (0,1)$ is a suitably chosen scaling factor \cite{Kanschat15}, while the action of the multiplicative Schwarz smoother is defined through:
\begin{equation}
\label{eqn:MSM}
    {\bf S}^{-1}_\ell = \left[ {\bf I} - \prod_i \left( {\bf I} - {\bf E}_{i,\ell}^T ({\bf E}_{i,\ell} {\bf K}_\ell {\bf E}_{i,\ell}^T)^{-1} {\bf E}_{i,\ell} {\bf K}_\ell \right) \right] {\bf K}^{-1}_\ell
\end{equation}
The additive and multiplicate Schwarz smoothers are generalizations of the classical Jacobi and Gauss-Seidel smoothers, and indeed they can be implemented in an efficient, iterative manner.  For both of these smoothers, a sequence of local matrix problems of the form:
\begin{equation}
\label{eqn:localproblem}
{\bf K}_{i,\ell} {\bf U}_{i,\ell} = {\bf F}_{i,\ell}
\end{equation}
where ${\bf K}_{i,\ell} = {\bf E}_{i,\ell} {\bf K}_\ell {\bf E}_{i,\ell}^T$ must be solved.  It is easily seen that:
\[
{\bf K}_{i,\ell} = \left[ \begin{array}{cc}
    {\bf A}_{i,\ell} & -{\bf B}_{i,\ell} \\
    {\bf B}_{i,\ell}^T & {\bf 0} 
\end{array} \right]
\]
where:
\begin{align}
{\bf A}_{i,\ell} &= {\bf E}^v_{i,\ell} {\bf A}_\ell \left({\bf E}^v_{i,\ell}\right)^T \nonumber \\
{\bf B}_{i,\ell} &= {\bf E}^v_{i,\ell} {\bf A}_\ell \left({\bf E}^q_{i,\ell}\right)^T \nonumber
\end{align}
Thus, with both the additive and multiplicative Schwarz smoothers, a discrete generalized Stokes or Oseen problem is solved for each subdomain.  In the next subsection, we further clarify this interpretation in a variational setting.  With the additive Schwarz smoother, the subdomain problems are solved independently, and their respective solutions are summed together and multiplied through by a scaling factor as indicated above.  With the multiplicative Schwarz smoother, the subdomain problems are solved in a sequential fashion in analogy with the Gauss-Seidel smoother.

\subsection{Preservation of the Divergence-free Constraint}

Now that we have presented our choice of smoother, we demonstrate that our geometric multigrid methodology preserves the divergence-free constraint on the velocity field.  Provided that the initial guess for the V-cycle algorithm satisfies the divergence-free constraint, it is sufficient to show that each smoothing step provides velocity updates that are divergence-free.  One application of either the additive or multiplicative Schwarz smoother at level $\ell$ is akin to solving a collection of local subdomain problems of the form: Find $\delta {\bf u}_{i,\ell} \in {\bf V}_{i,\ell}$ and $\delta p_{i,\ell} \in \text{Q}_{i,\ell}$ such that:
\begin{equation}
\label{eqn:subdomain}
a_h({\bf v}_h,\delta {\bf u}_{i,\ell}) - b({\bf v}_h,\delta p_{i,\ell}) + b(\delta {\bf u}_{i,\ell},q_h) = \ell({\bf v}_h) - a_h({\bf v}_h,{\bf u}_h) + b({\bf v}_h,p_h) - b({\bf u}_h,q_h) 
\end{equation}
for all ${\bf v}_h \in {\bf V}_{i,\ell}$ and $q_h \in \text{Q}_{i,\ell}$ where ${\bf u}_h$ and $p_h$ are the approximate discrete velocity and pressure solutions.  For the additive Schwarz smoother, the approximate discrete velocity and pressure solutions are updated following the solution of all of the local problems according to:
\[
\begin{aligned}
{\bf u}_h & \leftarrow {\bf u}_h + \eta \left(\sum_i \delta {\bf u}_{i,\ell}\right) \\
p_h & \leftarrow p_h + \eta \left(\sum_i \delta p_{i,\ell}\right)
\end{aligned}
\]
while for the multiplicative Schwarz smoother, the approximate discrete velocity and pressure solutions are updated following the solution of each individual local problem according to:
\[
\begin{aligned}
{\bf u}_h & \leftarrow {\bf u}_h + \delta {\bf u}_{i,\ell} \\
p_h & \leftarrow p_h + \delta p_{i,\ell}
\end{aligned}
\]
For each subdomain problem, if the approximate discrete velocity solution is divergence-free, it holds that:
\[
b(\delta {\bf u}_{i,\ell},q_h) = 0
\]
for all $q_h \in \text{Q}_{i,\ell}$.  Since $\text{Q}_{i,\ell} = \nabla \cdot {\bf V}_{i,\ell}$, we can select $q_h = \nabla \cdot \delta {\bf u}_{i,\ell}$ to find:
\[
\| \nabla \cdot \delta {\bf u}_{i,\ell} \|^2_{L^2(\Omega)} = b(\delta {\bf u}_{i,\ell},q_h) = 0
\]
and thus the solution to the local problem is also divergence-free.  Thus, if the initial guess for the V-cycle algorithm satisfies the divergence-free constraint, each subsequent application of the Schwarz smoother at any given level $\ell$ preserves the divergence-free constraint as well.

\subsection{Efficacy of the Structure-Preserving Geometric Multigrid Methodology}

We conclude here with a short discussion of the efficacy of our geometric multigrid methodology.  We restrict our discussion to the three-dimensional setting without loss of generality.  Recall that the spaces $\bm{\Psi}_{i,\ell}$, ${\bf V}_{i,\ell}$, and $\text{Q}_{i,\ell}$ form a discrete Stokes complex for a given level $\ell$ and subdomain $i$.  Thus, we can express the velocity solution $\delta {\bf u}_{i,\ell} \in {\bf V}_{i,\ell}$ to \eqref{eqn:subdomain} in terms of the curl of a vector potential $\delta \bm{\psi}_{i,\ell} \in \bm{\Psi}_{i,\ell}$ provided the velocity is divergence-free, and this vector potential can be obtained via the reduced subdomain problem: Find $\delta \bm{\psi}_{i,\ell} \in \bm{\Psi}_{i,\ell}$ such that:
\begin{equation}
\label{eqn:subdomain_reduced}
a_h(\nabla \times \bm{\zeta}_h,\nabla \times \delta \bm{\psi}_{i,\ell}) = \ell(\nabla \times \bm{\zeta}_h) - a_h(\nabla \times \bm{\zeta}_h,{\bf u}_h)
\end{equation}
for all $\bm{\zeta}_h \in \bm{\Psi}_{i,\ell}$.  This is precisely the subdomain problem associated with the global semi-elliptic generalized Maxwell problem with hyperresitivity \cite{Biskamp97}: Find $\bm{\psi} \in \bm{\Psi}_h$ such that:
\begin{equation}
\label{eqn:global_reduced}
a_h(\nabla \times \bm{\zeta}^h,\nabla \times \bm{\psi}^h) = \ell(\nabla \times \bm{\zeta}^h)
\end{equation}
for all $\bm{\zeta}^h \in \bm{\Psi}_h$.  It is known that a geometric multigrid methodology based on the use of Schwarz smoothers posed on structure-preserving subdomains is optimally convergent for Maxwell problems \cite{Arnold00}.  Consequently, we can expect that at least the discrete velocity solutions will converge in our approach.

\section{Numerical Results}
\label{sec:Results}

We now present a series of numerical tests illustrating the effectiveness of our proposed geometric multigrid methodology. Each of the tests correspond to problems with homogeneous Dirichlet boundary conditions applied along the entire domain boundary.  In our discretization scheme, no-penetration boundary conditions are enforced strongly and no-slip boundary conditions are enforced weakly using a penalty constant of $C_I = 4(p-1)$ where $p$ is the polynomial degree which is taken to be equal in each parameteric direction.  It should be noted that $p$ refers to the polynomial degree of the discrete streamfunction space in the two-dimensional case and the discrete scalar potential space in the three-dimensional case.  Hence, for $p = 2$, the discrete pressure fields are piecewise bilinear/trilinear B-splines rather than piecewise biquadratic/triquadratic B-splines.  Maximally smooth B-splines defined on uniform knot vectors are utilized throughout.

For all of the following tests, we define convergence as the number of V-cycles required to reduce the initial residual by a factor of $10^{6}$.  We always initialize the V-cycle algorithm using a random initial guess which satisfies the divergence-free constraint on the velocity field.  For each V-cycle, one pre-smoothing and two post-smoothing steps are employed using either the multiplicative or additive Schwarz smoother.  For the additive Schwarz smoother, a scaling factor of $\eta = 0.5$ is employed.

For all the problems presented here, a single element is used for the coarsest mesh and we investigate the convergence behavior for various levels of refinement.  We report on the convergence behavior of our method for both the generalized Stokes and Oseen problems as well as a selection of different problem parameters, polynomial degrees, domain geometries, and number of spatial dimensions (2D and 3D).  With respect to problem parameters, we consider the ratios between reaction and diffusion and advection and diffusion, which we express through a Damk\"ohler number ($Da$) and a Reynolds number ($Re$).  We define these numbers as:
\[
Da = \frac{\sigma L^2}{\nu} \quad \text{and} \quad Re = \frac{|{\bf a}|L}{\nu}.
\]
where $L$ is a characteristic length scale which is taken to be one throughout.

\subsection{Two-dimensional generalized Stokes flow in a square domain}

We first consider a two-dimensional generalized Stokes problem posed on the square domain $(0,1)^2$.  In particular, we consider a forcing:
\begin{equation}
\begin{aligned}
\label{eq:f_square}
\bf{f} = \sigma \bf{u} - \nu \Delta \bf{u} + \nabla p
\end{aligned}
\end{equation}
corresponding to the manufactured solution \cite{Buffa11}:
\begin{equation}
\begin{aligned}
\label{eq:u_square}
\bf{u} = \left[ \begin{array}{c}
	2e^x (-1 + x)^2 x^2 (y^2 - y)(-1 + 2y) \\
	(-e^x (-1 + x) x (-2 + x (3 + x))(-1 + y)^2 y^2 \end{array} \right]
\end{aligned}
\end{equation}
\begin{equation}
\begin{aligned}
\label{eq:p_square}
p & = (-424 + 156e + (y^2 - y) (-456 + e^x (456 + x^2 (228 - 5 (y^2 - y)) \\
& + 2x (-228 + (y^2 -y)) + 2x^3 (-36 + (y^2 - y)) + x^4 (12 + (y^2-1)))))
\end{aligned}
\end{equation}
The velocity field associated with this exact solution is plotted in Figure \ref{fig:2d_solution_square}.

\begin{figure}[b!]
\centering
\includegraphics[scale=0.33]{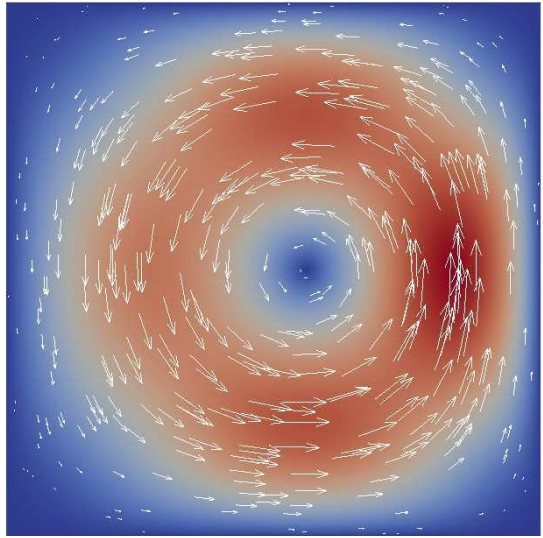}
\captionsetup{format = hang}
\caption{Velocity field for the unit square generalized Stokes problem.}
\label{fig:2d_solution_square}
\end{figure}

\begin{table}[h]
\caption{Number of V(1,2) cycles required for convergence for the unit square generalized Stokes problem using the multiplicative Schwarz smoother and $p = 2, 3$.}
\label{table:2D-DS-square}
\centering
\tabsize
\begin{tabular}{llcclcc}
\toprule
 & \multicolumn{3}{c}{$p = 2$} & \multicolumn{3}{c}{$p = 3$} \\ 
 \midrule
$n_\ell$ & DOFs & $Da = 1$ & $Da = 1000$ & DOFs & $Da = 1$ & $Da = 1000$ \\
\midrule
1 & 16 & 3 & 2 & 25 & 6 & 6 \\
2 & 36 & 5 & 4 & 49 & 12 & 14 \\
3 & 100 & 5 & 5 & 121 & 13 & 13 \\
4 & 324 & 6 & 5 & 361 & 12 & 12 \\
5 & 1156 & 6 & 5 & 1225 & 13 & 13 \\
6 & 4356 & 7 & 6 & 4489 & 13 & 13 \\
7 & 16900 & 7 & 6 & 17161 & 13 & 13 \\
8 & 66564 & 7 & 6 & 67081 & 13 & 13 \\
9 & 264196 & 7 & 7 & 265225 & 13 & 13 \\
10 & 1052676 & 7 & 7 & 1054729 & 13 & 13 \\
\bottomrule
\end{tabular}
\end{table}

\begin{table}[h]
\caption{Number of V(1,2) cycles required for convergence for the unit square generalized Stokes problem using the additive Schwarz smoother and $p = 2$.}
\label{table:2D-DS-square-AS}
\centering
\tabsize
\begin{tabular}{llcc}
\toprule
$n_\ell$ & DOFs & $Da = 1$ & $Da = 1000$ \\
\midrule
1 & 16 & 5 & 4 \\
2 & 36 & 12 & 9 \\
3 & 100 & 15 & 14 \\
4 & 324 & 16 & 15 \\
5 & 1156 & 16 & 16 \\
6 & 4356 & 16 & 16 \\
7 & 16900 & 16 & 16 \\
8 & 66564 & 16 & 16 \\
9 & 264196 & 16 & 16 \\
10 & 1052676 & 16 & 16 \\
\bottomrule
\end{tabular}
\end{table}

We first present convergence results for our multigrid method using the multiplicative Schwarz smoother and polynomial degrees $p = 2$ and $p = 3$ in Table \ref{table:2D-DS-square}.  It is clear that for a given polynomial order, the convergence behavior is robust with respect to both the number of levels of refinement and the problem parameters.  We also observe that as polynomial order is increased, the convergence behavior deteriorates, albeit slightly.  This is consistent with previously observed behavior for isogeometric analysis \cite{Gahalaut13, Hofreither15, Hofreither17}.

We next present convergence results for our multigrid method using the additive Schwarz smoother and polynomial degree $p = 2$ in Table \ref{table:2D-DS-square-AS}.  As expected, overall convergence is slower with additive Schwarz than with multiplicative Schwarz, although in this case the method is still robust with regard to both the number of levels of refinement and the problem parameters.

\subsection{Two-dimensional generalized Stokes flow in a quarter annulus}

We next consider a two-dimensional generalized Stokes problem posed on a quarter annulus.  The domain is described in Figure \ref{fig:annulus} where $r_i = 0.075$ and $r_o = 0.225$.  We consider a manufactured solution achieved by mapping the solution presented in \eqref{eq:u_square}-\eqref{eq:p_square} to the quarter-annulus domain using a quadratic rational B\'{e}zier parametric mapping and appropriate push-forward operators.  The velocity field associated with the exact solution are also plotted in Figure \ref{fig:annulus}.

\begin{figure}[b!]
\centering
\includegraphics[scale=0.33]{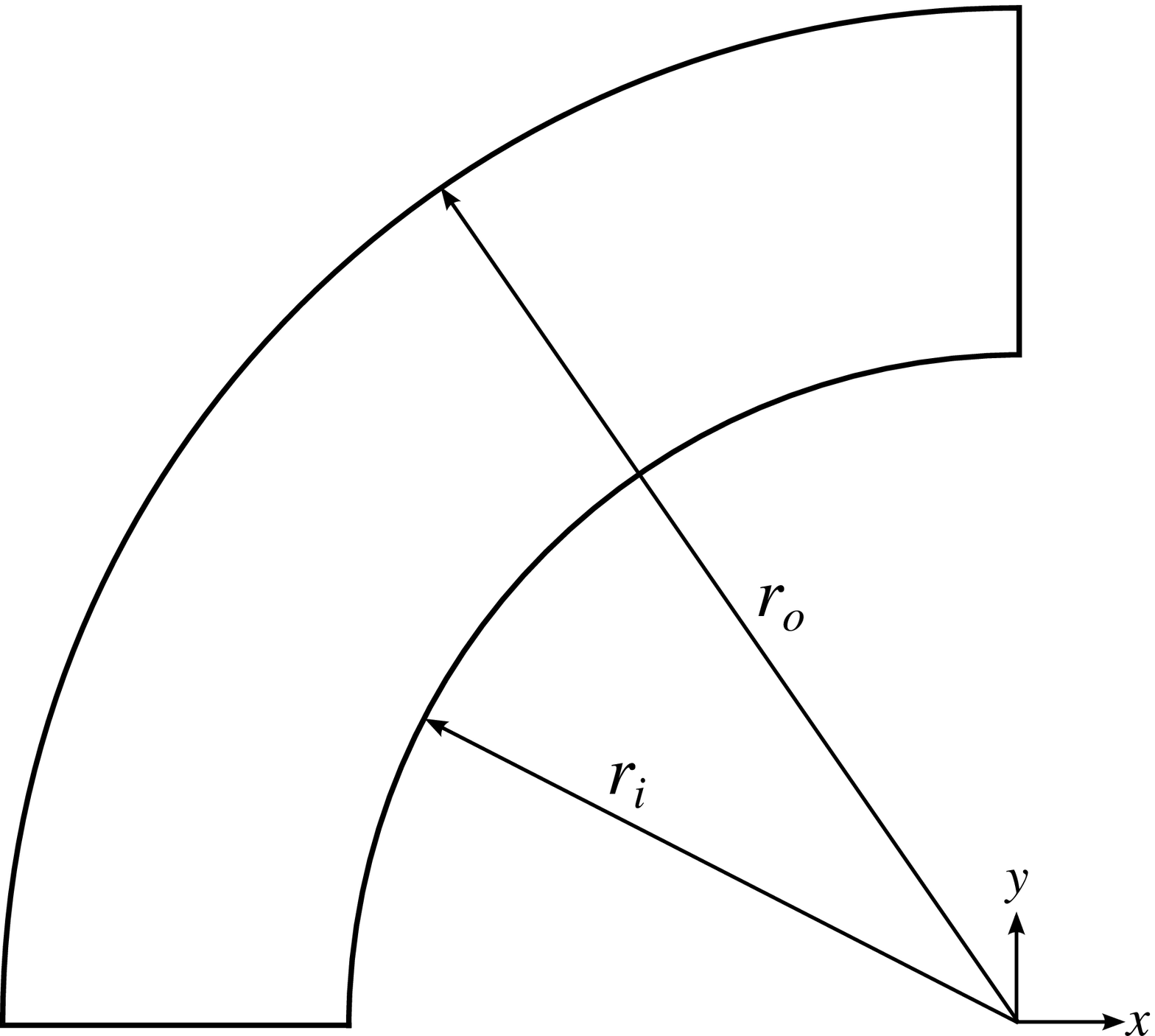} \includegraphics[scale=0.26]{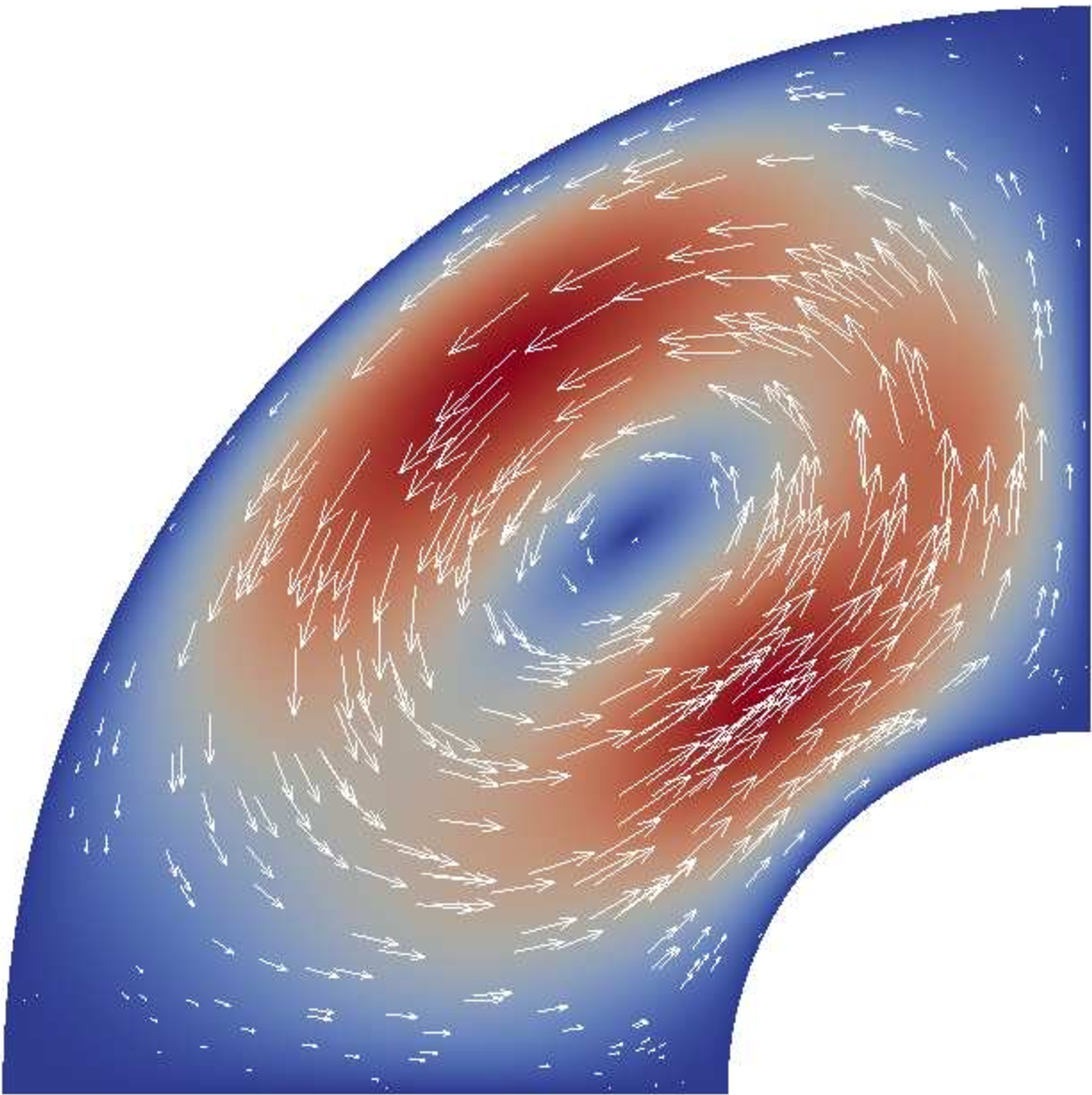}
\captionsetup{format = hang}
\caption{Domain description (left) and velocity field (right) for the quarter annulus generalized Stokes problem.}
\label{fig:annulus}
\end{figure}

We present convergence results for our multigrid method using the multiplicative Schwarz smoother and polynomial degree $p = 2$ in Table \ref{table:2D-DS-annulus}.  Compared with the square domain, the number of V-cycles required for convergence is larger.  However, the method is still robust with respect to both the number of levels of refinement and the problem parameters.

\begin{table}[t!]
\caption{Number of V(1,2) cycles required for convergence for the quarter annulus generalized Stokes problem using the multiplicative Schwarz smoother and $p = 2$.}
\label{table:2D-DS-annulus}
\centering
\tabsize
\begin{tabular}{llcc}
\toprule
$n_\ell$ & DOFs & $Da = 1$ & $Da = 1000$ \\
\midrule
1 & 16 & 2 & 2 \\
2 & 36 & 3 & 3 \\
3 & 100 & 8 & 7 \\
4 & 324 & 16 & 15 \\
5 & 1156 & 24 & 24 \\
6 & 4356 & 26 & 26 \\
7 & 16900 & 31 & 31 \\
8 & 66564 & 34 & 34 \\
9 & 264196 & 34 & 34 \\
10 & 1052676 & 34 & 34 \\
\bottomrule
\end{tabular}
\end{table}

\subsection{Three-dimensional generalized Stokes flow in a cube domain}
\begin{figure}[b!]
\centering
\includegraphics[scale=0.2]{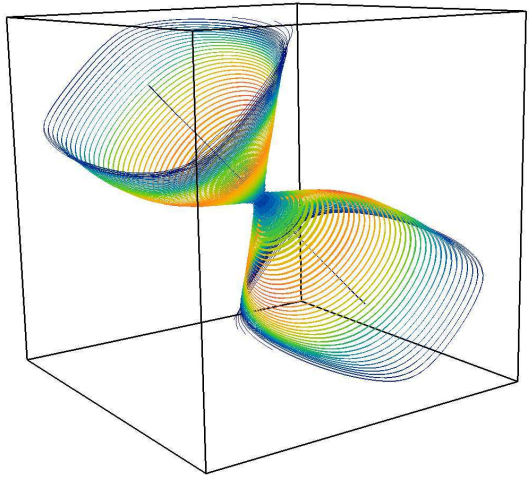}
\captionsetup{format = hang}
\caption{Streamlines colored by velocity magnitude for the unit cube generalized Stokes problem.}
\label{fig:3d_solution_cube}
\end{figure}

We next consider a three-dimensional generalized Stokes problem posed on the unit cube $(0,1)^3$.  In particular, we consider a forcing:
\begin{equation}
\begin{aligned}
\label{eq:f_cube}
\bf{f} = \sigma \bf{u} - \nu \Delta \bf{u} + \nabla p
\end{aligned}
\end{equation}
corresponding to the manufactured solution \cite{Evans13_1}:
\begin{equation}
\begin{aligned}
\label{eq:u_cube}
\bf{u} = \nabla \times \bm{\psi}
\end{aligned}
\end{equation}
\begin{equation}
\begin{aligned}
\label{eq:phi_cube}
\bm{\psi} = \left[ \begin{array}{c}
	x(x-1)y^2(y-1)^2z^2(z-1)^2 \\
	0 \\
	x^2(x-1)^2y^2(y-1)^2z(z-1) \end{array} \right]
\end{aligned}
\end{equation}
\begin{equation}
\begin{aligned}
\label{eq:p_cube}
p = \sin(\pi x) \sin(\pi y) - \frac{4}{\pi^2}
\end{aligned}
\end{equation}
Streamlines colored by velocity magnitude associated with the exact solution are plotted in Figure \ref{fig:3d_solution_cube}.

We present convergence results for our multigrid method using the multiplicative Schwarz smoother and polynomial degrees $p = 2$ and $p = 3$ in Table \ref{table:3D-DS-cube}.  Convergence appears to be much quicker in the three-dimensional setting.  Notably, one V-cycle appears to be sufficient to reduce the residual by six orders of magnitude for a sufficient number of levels for both $p = 2$ and $p = 3$ and irrespective of the Damk\"ohler number.  We believe this may be due to the fact each velocity degree of freedom is updated twice as many times in each iteration of the smoother in the three-dimensional case as compared to the two-dimensional case.

\begin{table}[t!]
\caption{Number of V(1,2) cycles required for convergence for the unit cube generalized Stokes problem using the multiplicative Schwarz smoother and $p = 2, 3$.}
\label{table:3D-DS-cube}
\centering
\tabsize
\begin{tabular}{llcclcc}
\toprule
& \multicolumn{3}{c}{$p = 2$} & \multicolumn{3}{c}{$p = 3$} \\ 
 \midrule
$n_\ell$ & DOFs & $Da = 1$ & $Da = 1000$ & DOFs & $Da = 1$ & $Da = 1000$ \\
\midrule
1 & 144 & 1 & 1 & 300 & 2 & 4 \\
2 & 540 & 1 & 1 & 882 & 1 & 1 \\
3 & 2700 & 2 & 2 & 3630 & 1 & 1 \\
4 & 16524 & 1 & 1 & 19494 & 1 & 1 \\
5 & 114444 & 1 & 1 & 124950 & 1 & 1 \\
6 & 849420 & 1 & 1 & 888822 & 1 & 1 \\
\bottomrule
\end{tabular}
\end{table}

\subsection{Three-dimensional generalized Stokes flow in a hollow cylinder section}
The final generalized Stokes problem considered in this paper is a three-dimensional problem posed on a hollow cylinder section.  The domain for this problem is simply the quarter annulus from before extruded in the $z$-direction by a depth of $d = 0.1$.  We consider a manufactured solution achieved by mapping the solution presented in \eqref{eq:u_cube}-\eqref{eq:p_cube} to the hollow cylinder section domain using a quadratic rational B\'{e}zier parametric mapping and appropriate push-forward operators. Streamlines colored by velocity magnitude associated with the exact solution are plotted in Figure \ref{fig:3d_solution_annulus}.

\begin{figure}[b!]
\centering
\includegraphics[scale=0.2]{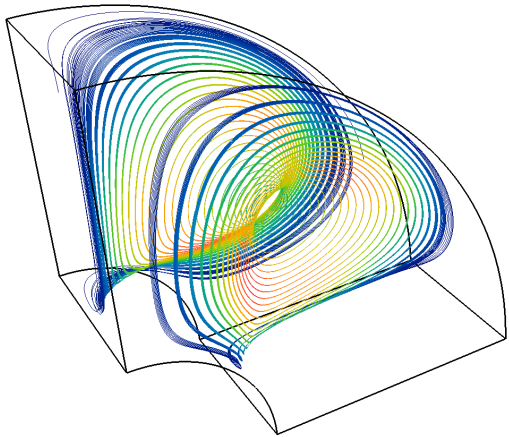}
\captionsetup{format = hang}
\caption{Streamlines colored by velocity magnitude for the hollow cylinder generalized Stokes problem.}
\label{fig:3d_solution_annulus}
\end{figure}

We present convergence results for our multigrid method using the multiplicative Schwarz smoother and polynomial degree $p = 2$ in Table \ref{table:3D-DS-hollow}.  Incredibly, one V-cycle again appears to be sufficient to reduce the residual by six orders of magnitude for a sufficient number of levels irrespective of the Damk\"ohler number.

\begin{table}[t!]
\caption{Number of V(1,2) cycles required for convergence for the hollow cylinder generalized Stokes problem using the multiplicative Schwarz smoother and $p = 2$.}
\label{table:3D-DS-hollow}
\centering
\tabsize
\begin{tabular}{llcc}
\toprule
$n_\ell$ & DOFs & $Da = 1$ & $Da = 1000$\\
\midrule
1 & 144 & 1 & 1 \\
2 & 540 & 1 & 2 \\
3 & 2700 & 2 & 2 \\
4 & 16524 & 1 & 1 \\
5 & 114444 & 1 & 1 \\
6 & 849420 & 1 & 1 \\
\bottomrule
\end{tabular}
\end{table}

\subsection{Two-dimensional generalized Oseen flow in a square domain}

We now turn our attention to the generalized Oseen problem.  We first consider a two-dimensional generalized Oseen problem posed on the square domain $(0,1)^2$.  We manufacture a solution with a forcing:
\begin{align}
\bf{f} = \sigma \bf{u} + \bf{a} \cdot \nabla \bf{u} - \nu \Delta \bf{u} + \nabla p
\end{align}
where ${\bf u}$ and $p$ are defined as in \eqref{eq:u_square}-\eqref{eq:p_square} such that the resulting solution is the same as the unit square generalized Stokes problem.  Note that the advection velocity is taken to be the manufatured velocity field.

We present convergence results for our multigrid method using the multiplicative Schwarz smoother, polynomial degree $p = 2$, and various Reynolds and Damk\"ohler numbers in Table \ref{table:2D-Oseen}.  When Reynolds number is low, the advection terms become negligible, and thus the method performs as it did on the the 2D generalized Stokes problem.  As the Reynolds number is increased, the advection term becomes more significant.  In this case, we have observed favorable convergence behavior as long as the Damk\"ohler number is at least as large as the Reynolds number.  When the system becomes advection-dominated, on the other hand, the multigrid method fails to converge.  We expect that improved results may be obtained through the use of an alternative smoother which respects the directionality of the advection velocity.

\subsection{Three-dimensional generalized Oseen flow in a cube domain}

We conclude by considering a three-dimensional generalized Oseen problem posed on the unit cube $(0,1)^3$.  We manufacture a solution with a forcing:
\begin{align}
\bf{f} = \sigma \bf{u} + \bf{a} \cdot \nabla \bf{u} - \nu \Delta \bf{u} + \nabla p
\end{align}
where ${\bf u}$ and $p$ are defined as in \eqref{eq:u_cube}-\eqref{eq:p_cube} such that the resulting solution is the same as the unit cube generalized Stokes problem.  As with the two-dimensional generalized Oseen problem, the advection velocity is take to be the manufactured velocity field.

We present convergence results for our multigrid method using the multiplicative Schwarz smoother, polynomial degree $p = 2$, and various Reynolds and Damk\"ohler numbers in Table \ref{table:3D-Oseen}.  The same trends that were observed for the two-dimensional case are observed here as well.  Namely, when the system is not advection dominated, we achieve excellent convergence behavior.  Also, as was the case with the generalized Stokes flow, the three-dimensional case exhibits improved convergence as compared to the two-dimensional case.

\begin{table}[t]
\caption{Number of V(1,2) cycles required for convergence for the unit square generalized Oseen problem using the multiplicative Schwarz smoother and $p = 2$.}
\label{table:2D-Oseen}
\centering
\tabsize
\begin{tabular}{llccc}
\toprule
 & & \multicolumn{2}{c}{$Re = 1$} & $Re = 100$ \\ 
\midrule
$n_\ell$ & DOFs & $Da = 1$ & $Da = 1000$ & $Da = 1000$\\
\midrule
1 & 16 & 3 & 2 & 2 \\
2 & 36 & 5 & 4 & 3 \\
3 & 100 & 5 & 5 & 4 \\
4 & 324 & 6 & 5 & 9 \\
5 & 1156 & 6 & 5 & 13 \\
6 & 4356 & 7 & 6 & 15 \\
7 & 16900 & 7 & 6 & 11 \\
8 & 66564 & 7 & 6 & 7 \\
9 & 264196 & 7 & 7 & 7 \\
10 & 1052676 & 7 & 7 & 7 \\
\bottomrule
\end{tabular}
\end{table}

\begin{table}[!h]
\caption{Number of V(1,2) cycles required for convergence for the unit cube generalized Oseen problem using the multiplicative Schwarz smoother and $p = 2$.}
\label{table:3D-Oseen}
\centering
\tabsize
\begin{tabular}{llccc}
\toprule
& & \multicolumn{2}{c}{$Re = 1$} & $Re = 100$ \\ 
\midrule
$n_\ell$ & DOFs & $Da = 1$ & $Da = 1000$ & $Da = 1000$\\
\midrule
1 & 144 & 1 & 2 & 2 \\
2 & 540 & 1 & 2 & 2 \\
3 & 2700 & 2 & 3 & 3 \\
4 & 16524 & 1 & 2 & 2 \\
5 & 114444 & 1 & 1 & 3 \\
6 & 849420 & 1 & 1 & 3 \\
\bottomrule
\end{tabular}
\end{table}

\section{Conclusions}
\label{sec:Conclusion}

In this paper, we presented a structure-preserving geometric multigrid methodology for isogeometric compatible discretizations of the generalized Stokes and Oseen problems which relies upon Schwarz-style smoothers in conjunction with specially chosen subdomains.  We proved that our methodology yields a pointwise divergence-free velocity field independent of the number of pre-smoothing steps, post-smoothing steps, grid levels, or cycles in a V-cycle implementation, and we demonstrated the efficiency and robustness of our methodology by numerical example.  Specifically, we found that our methodology exhibits convergence rates independent of the grid resolution and flow parameters for the generalized Stokes problem as well as the generalized Oseen problem provided it is not advection-dominated.  We also discovered that, somewhat surprisingly, our methodology exhibits improved convergence rates in the three-dimensional setting as compared with the two-dimensional setting.

We envision several avenues for future work.  First of all, we plan to conduct a full mathematical analysis of our methodology.  We anticipate that this analysis will largely follow the same program of work as laid out in a recent geometric multigrid paper for divergence-conforming discontinuous Galerkin formulations of Stokes flow \cite{Kanschat15}.  Second, we would like to extend the applicability of our methodology to advection-dominated Oseen problems.  We anticipate the need for upwind-based line smoothers in such a setting \cite{Mavriplis02}.  Third, we plan to extend our methodology to multi-patch geometries and adaptive isogeometric compatible discretizations \cite{Buffa14,Johannessen15}.  Initial results in this area are quite encouraging.  Finally, we plan to extend our methodology to multi-physics problems, including coupled flow transport, fluid-structure, and magnetohydrodynamics.

\section{Acknowledgement}
This material is based upon work supported by the Air Force Office of Scientific Research under Grant No. FA9550-14-1-0113.

\bibliographystyle{wileyj}
\bibliography{references}

\end{document}